\documentclass[12pt]{amsart}  
\usepackage{graphicx}
\usepackage{indentfirst}
\usepackage{amsmath,amscd}
\usepackage{amsfonts}
\usepackage{graphics}
\usepackage{amssymb}
\usepackage{enumerate}
\usepackage[usenames,dvipsnames]{pstricks}
\usepackage{epsfig}
\usepackage{pst-grad} 
\usepackage{pst-plot} 
\usepackage[all]{xy}

\newtheorem{thm}{Theorem}
\newtheorem{cor}[thm]{Corollary}

\newtheorem{theorem}{Theorem}
\newtheorem{lemma}[thm]{Lemma}
\newtheorem{proposition}[thm]{Proposition}
\theoremstyle{remark}
\newtheorem{remark}[theorem]{Remark}

\numberwithin{equation}{subsubsection}
\numberwithin{thm}{subsection}

\newcommand\C{\mathcal{C}} 
\newcommand\N{\mathbb{N}}
\newcommand\R{\mathbb{R}}

\DeclareMathOperator{\Hom}{Hom}
\DeclareMathOperator{\Tr}{Tr}
\DeclareMathOperator{\tr}{tr}

\DeclareMathOperator{\SU}{SU}

\newcommand\X{\mathfrak{X}}

\begin{document}

\title[Ergodicity on moduli spaces of non-orientable surfaces]{Ergodic actions of mapping class groups on moduli spaces of representations of non-orientable surfaces}
\author{Frederic Palesi}
\address{Institut Fourier, UMR5582, UFR Mathématiques, Universit\'e Grenoble I\\
100 rue des Maths BP74, 38402 St Martin d'Heres, France} \email{fpalesi@ujf-grenoble.fr} 
\date{\today}
\begin{abstract}Let $M$ be a non-orientable surface with Euler characteristic $\chi (M) \leq -2$.  We consider the moduli space of flat $\SU(2)$-connections, or equivalently the space of conjugacy classes of representations $$\X (M) = \Hom (\pi_1 (M) , \SU (2)) / \SU (2) .$$
There is a natural action of the mapping class group of $M$ on $\X (M)$. We show here that this action is ergodic with respect to a natural measure. This measure is defined using the push-forward measure associated to a map defined by the presentation of the surface group. This result is an extension of earlier results of Goldman for orientable surfaces (see \cite{go1}).
\end{abstract}
\maketitle


\section{Introduction}
Let $M$ be a closed surface with $\chi (M) < 0$ and let $\pi$ denote its fundamental group. Let $G$ be a Lie group and consider the space of homomorphisms of $\pi$ into $G$, denoted $\Hom (\pi , G)$, and called a {\it representation variety}. The space of $G$-conjugacy classes of such homomorphisms, called the {\it character variety} is denoted $\X (M) = \Hom (\pi , G ) / G$. Geometrically, $\X (M)$ is the moduli space of flat principal $G$-bundles over $M$.

The mapping class group, denoted $\Gamma_M$, is defined as the group of isotopy classes of orientation-preserving diffeomorphisms of $M$ when the surface is orientable, and as the whole group of isotopy classes of diffeomorphisms when the surface is non-orientable. A classical result of Nielsen \cite{nielsen} tells us that the mapping class group of an orientable surface is isomorphic to the group $\mbox{Out}^+ (\pi)$ of positive outer automorphisms of $\pi$. When the surface is non-orientable, Mangler \cite{mangler} proved that the mapping class group is isomorphic to the full group $\mbox{Out} (\pi)$. Hence, in both cases, there is a natural action of $\Gamma_M$ on $\X (M)$ induced by the action of $\mbox{Aut} (\pi) \times \mbox{Aut} (G)$ on $\Hom (\pi , G)$ by left and right composition.

The dynamics of these actions have been extensively studied in the orientable case for various Lie groups (see \cite{goldman06} for a survey on the subject). In the following, we will focus on the case of a non-orientable surface with $G = \SU(2)$, which allows explicit calculations with trace coordinates.

When $M$ is orientable, there is a natural $\Gamma_M$-invariant symplectic structure on $\X (M)$ (see \cite{goldman84,goldman04}) which induces a volume form and hence a measure. In \cite{go1} Goldman uses this symplectic structure and a certain Hamiltonian $\R^n$-action defined on the character variety to show that the $\Gamma_M$-action is ergodic on $\X(M)$ with respect to this measure. However, when the surface $M$ is non-orientable, a symplectic structure may not exist on the character variety as the dimension of this space might be odd. So another approach is necessary to define a measure on $\X (M)$ in the non-orientable case.
In  \cite{witten} Witten defines and computes a volume on $\X (M)$ using the Reidemeister-Ray-Singer torsion (see e.g. \cite{freed}). In the case of an orientable surface, Witten proves that this volume equals the symplectic volume on the moduli space. In \cite{ho} Jeffrey and Ho prove that Witten's volume arises from the Haar measure, in the case of a non-orientable surface. In \cite{mulase} Mulase and Penkava compute the volume of the representation space using a certain volume distribution given by the push-forward measure associated to a presentation map of $\pi_1 (M)$ and their formula also agreed with Witten's result. Using this point of view, we define a $\Gamma_M$-invariant measure on the moduli space, denoted $\nu$.

The main result of this paper is the following:
\begin{theorem}\label{closed}
Let $M$ be a closed non-orientable surface such that $\chi (M) \leq - 2$ and let $G = \SU (2)$. Then the mapping class group $\Gamma_M$ acts ergodically on $\X (M)$ with respect to $\nu$.
\end{theorem}
The analogous result for an orientable surface was proved by Goldman in \cite{go1}. 
In order to prove Theorem \ref{closed} we need to consider subsurfaces (with boundary) of $M$ and it will be useful to consider a more general version of this result for surfaces with boundary. Assume $M$ has $m$ boundary components denoted $\partial_1 M , ... , \partial_m M$. The inclusion maps $\partial_i M \hookrightarrow M$ induce the application
$$\partial ^{\#} : \X(M) \longrightarrow \X (\partial M) : = \prod_{i=1}^m \X (\partial_i M) .$$

Then $\X(M)$ can be viewed as a family of {\it relative character varieties} over $\X (\partial M)$. As each $\X (\partial_i M) $ is identified to the set $[ G ]$ of conjugacy classes in $G$, the base of this family is a product of copies of $[G]$. Specifically, let $ \{ C_1 , C_2 , ... , C_m \} $ be a set of elements of the fundamental group $\pi$ corresponding to the $m$ boundary components. Let $\C = (c_1 , ... , c_m)$ be an element of $[G]^m$, and define the {\it relative character variety} over $\C $ as 
$$\X_{\C} (M) = \partial_{\#}^{-1} (\C) = \{ [ \rho ] \in \X (M) \mid [ \rho (C_i) ] = c_i , 1 \leq i \leq m \}. $$
The disintegration of the measure $\nu$ on $\X (M)$ with respect to $\partial_{\#}$ is a measure $\nu_{\C}$ on the submanifold $\partial_{\#}^{-1} (\C)$.

For a surface with boundary, the mapping class group $\Gamma_M$ is identified with the group $\mbox{Out} (\pi , \partial M)$ of outer automorphisms of $\pi$ (respectively $\mbox{Out}^+ (\pi , \partial M)$, if the surface is orientable) which preserve the conjugacy class of every cyclic subgroup corresponding to a boundary component. Then $\Gamma_M$ acts on $\X (M)$ by outer automorphisms of $\pi$ which preserve the function $\partial_{\#}$. Hence $\Gamma_M$ acts on $\X_{\C} (M)$, for every $\C \in [G]^m$. The generalization of Theorem \ref{closed} is the following:

\begin{theorem}\label{thm:open}
Let $M$ be a compact non-orientable surface with $m$ boundary components such that $\chi (M) \leq -2$ and let $\C = (c_1 , ... , c_m) \in [G]^m$. Then the action of the mapping class group $\Gamma_M$ on $\X_{\C} (M)$ is ergodic with respect to the measure $\nu_{\C}$.
\end{theorem}

This theorem includes Theorem \ref{closed} as the special case where $M$ has no boundary. The similar result for orientable surfaces was also proved by Goldman in \cite{go1}.

\begin{remark} For surfaces of Euler characteristic $-1$, the behavior of the $\Gamma_M$-action depends on the orientability: 
\begin{itemize}
\item If $M$ is orientable, namely a three-holed sphere or a one-holed torus, then the action of the mapping class group is ergodic on the relative character variety.
\item If $M$ is non-orientable, namely the two-holed projective plane, the one-holed Klein bottle or the connected sum of three projective planes, then the action of $\Gamma_M$ is not ergodic. In each of these cases, there is an essential curve which is invariant under the action of the mapping class group (see \cite{gendulphe}).
\end{itemize}
\end{remark}

\begin{remark}
For an orientable surface, the analog of Theorem \ref{thm:open} was extended to the general case of a compact Lie group $G$ by Pickrell and Xia in \cite{picxia1,picxia2}. Their approach relies on the study of the infinitesimal transitivity in the case of the one-holed torus and afterwards using sewing techniques on the representation variety. For non-orientable surfaces, one can expect that a similar result holds. However, as we can not have ergodicity for surfaces of Euler characteristic $-1$, we would have to study this infinitesimal transitivity in the case of the two-holed Klein bottle and the three-holed projective plane, which involve much more technical complications.
\end{remark}

\begin{remark}
The topological dynamics of these actions are more delicate as we do not ignore the subsets of null measure. We can hope that if a representation $\rho \in \Hom (\pi , \SU (2) ) $ has dense image in $\SU (2)$, then the $\Gamma_M$-orbit of $[\rho]$ is dense in $\X(M)$. This result is true if the surface $M$ is orientable and the genus of $M$ is strictly positive (\cite{prexia1,prexia2}). However in genus 0, there are representations $\rho$ with dense image but whose orbit $\Gamma_M \cdot [ \rho ]$ consists only of two points (see \cite{prexia3}). 
\end{remark}

\subsection*{Summary}

This paper is organized as follows.

In Section \ref{section:prelim}, we review some basic knowledge about non-orientable surfaces, their mapping class groups and moduli spaces. In Section \ref{section:mesure}, we define the $\Gamma_M$-invariant measure on $\X(M)$ using a certain volume distribution and its character expansion.

In Section \ref{section:flot}, we define the {\it Goldman flow} on non-orientable surfaces following Klein \cite{klein}. This is a circle action on a dense open subspace of the character variety of $M$. This action corresponds to the circle action defined by L. Jeffrey and J. Weitsman in \cite{jefwei} in the case of an orientable closed surface. This flow is related to a particular decomposition of the surface along a curve. In particular, the Dehn twist along this curve acts as a rotation on the orbit of the flow.

In Section \ref{section:pair} we study the case where $M$ is a non-orientable surface of even genus. In this case we split $M$ along a non-separating 2-sided curve $X$ to obtain an orientable surface $A$ with two additional boundary components. The orbits of the Goldman flow associated to $X$ are the fibers of the map $\X(M) \rightarrow \X (A)$. The Dehn twist about $X$ acts as a rotation on this fiber, and for almost all representation this rotation is irrational and hence ergodic. We infer that a $\Gamma_M$-invariant function on $\X (M)$ depends only on its value on $\X (A)$. Then the ergodicity in the case of an orientable surface proves that the $\Gamma_M$-invariant function depends only on its value at $X$. Then consider an embedding of a two-holed Klein bottle inside $M$, such that $X$ cuts open the two-holed Klein bottle into a four-holed sphere. We can find trace coordinates on the character variety of the two-holed Klein bottle. The explicit calculations for the action of a certain Dehn twist in these coordinates, allow us to settle the Theorem \ref{thm:open} in the case of a two-holed Klein bottle. In particular, this shows that a $\Gamma_M$-invariant function on $\X(M)$ does not depend on its value at $X$, which proves the theorem.

If $M$ is a non-orientable surface of odd genus, then it is impossible to cut open $M$ along a 2-sided curve into one or two orientable surfaces. Instead of that, we split $M$ along a separating curve $C$ into two parts denoted $A$ and $B$, such that $A$ is an orientable surface and $B$ is a non-orientable surface of Euler characteristic $-2$. The surface $B$  can be of two kind, a three-holed projective plane or a one-holed non-orientable surface of genus $3$. For these surfaces, we use trace cordinates to make explicit calculations for the action of Dehn twists. These calculations are contained in Section \ref{section:-2} and setlle the Theorem \ref{thm:open} in the case of a non-orientable surface of odd genus with Euler characteristic $-2$. 

In Section \ref{section:conclusion}, we use the Goldman flow associated to the separating curve $C$ to show that the Dehn twist about $C$ acts as a rotation on the fiber of the map $\X(M) \rightarrow \X(A) \times \X (B)$. For almost all representation, this rotation is ergodic. Hence, a $\Gamma_M$-invariant depends only on its value at $C$. Finally we consider an embedding of a four-holed sphere into $M$ such that $C$ is a separating non-trivial curve in it. The ergodicity for the four-holed sphere allows us to prove the theorem.

\section{Preliminaries}\label{section:prelim}

\subsection{Non-orientable surfaces}
We summarize some basic notions and results about on non-orientable surfaces and their mapping class groups. For more details and proofs, we refer to \cite{korkmaz,lickorish,mangler,stukow}.\\

Let $M $ be a compact non-orientable surface of genus $g \geq 1$ and with $m$ boundary components, denoted $N_{g,m}$. The boundary components of $M$ are denoted $$\partial M = C_1 \sqcup ... \sqcup C_m .$$
Recall  that $N_{g,0}$ is a connected sum of $g$ projective planes, and that $N_{g,m}$ is obtained by removing $m$ open disks of $N_{g,0}$. The fundamental group $\pi_1 ( N_{g,0} )$ admits two important presentations that we recall here. The first is the natural presentation which exhibits the fact that $N_g$ is a connected sum of projective planes.
$$\pi_1(N_{g,0}) = \langle A_1 , ... , A_g \, \mid \, A_1^2 \cdots A_g^2  \rangle .$$
Another presentation can be obtained by making use of the homeomorphism between the connected sum of three projective planes and the connected sum of a torus with one projective plane. The presentation depends on the parity of $g$:
$$\pi_1 ( N_{2k +1,0}) = \langle A_1 , B_1 , ... , A_k , B_k , C \, \mid \, [A_1 , B_1] ... [A_k , B_k] C^2   \rangle $$
$$\pi_1 ( N_{2k +2,0}) = \langle A_1 , B_1 , ... , A_k , B_k , C ,D \, \mid \, [A_1 , B_1] ... [A_k , B_k] C^2 D^2 \rangle . $$

A simple closed curve on a surface $M$ is called {\it two-sided} if a regular neighborhood of it within $M$ is homeomorphic to an annulus. A simple closed curve is called {\it one-sided} if a regular neighborhood of it within $M$ is homeomorphic to a M\" obius strip. A {\it circle} on $M$ is a closed connected one-dimensional submanifold of $M$. We denote by $M | X$ the surface obtained by cutting open $M$ along a circle $X$, defined as the surface with boundary for which there is an identification map $i_X : M | X \rightarrow M$ satisfying 

$\bullet$ the restriction of $i_X$ to $i_X^{-1} (M - X)$ is a diffeomorphism;

$\bullet$ $i_X^{-1} (X)$ consists of two components $X_+ , X_- \subset \partial (M|X)$, to each of which the restriction of $i_X$ is a diffeomorphism onto $X$.

A circle $X$ is called {\it non-separating} if $M|X$ is connected, and {\it separating} otherwise. A separating circle is {\it trivial} if one of the two components is either a disk, a cylinder or a M\" obius strip.\\

\subsection{Mapping class groups of non-orientable surfaces}

The mapping class group $\Gamma_M$ is defined to be the group of isotopy classes of diffeomorphisms $\phi : M \longrightarrow M$ which restrict to the identity on each boundary component, {\it i.e.} $\phi_{\mid C_i } = Id_{\mid C_i }$ for all $i$. Let $X$ be a two-sided circle on $M$, and let $U$ be a regular neighborhood of $X$ within $M$. The annulus $U$ is homeomorphic to $\mathbb{S}^1 \times [0,1]$, and we chose coordinates $(s,t)$ on this annulus. Let $f$ be the diffeomorphism of $M$ that is the identity outside of $U$, and that is defined inside $U$ as
$$f (s , t) = (s e^{2 i \pi t} , t).$$
The isotopy class of this map is called the {\it Dehn twist} about $X$, denoted $\tau_X$. Observe that this definition does not make sense for a one-sided curve.

For an orientable surface $S$, the mapping class group $\Gamma_S$ is generated by Dehn twists, and the number of generators can be chosen to be finite (see e.g. \cite{lickorish2,nielsen}). For a non-orientable surface $M$, the Dehn twists generate an index 2 subgroup of $\Gamma_M$, called the {\it twist subgroup} of $M$. Henceforth in this case, we need to define another family of diffeomorphisms of $M$ to find a generating set for $\Gamma_M$.
  
Consider a M\" obius strip $M$ with one hole, or equivalently a projective plane from which the interiors of two disks have been removed. Attach another M\" obius strip $N$ along one of the boundary components. The resulting surface $K$ is a Klein bottle with one hole. By sliding $N$ once along the core of $M$, we get a diffeomorphism $y_K$ of $K$ fixing the boundary of $K$ (cf. the Figure 1 below). Assume that this diffeomorphism is the identity in a neighborhood of the boundary of $K$. If $K$ is embedded in a surface $S$, we define $y$ as the diffeomorphism of $S$ that is the identity outside of $K$ and is given by $y_K$ inside $K$. The isotopy class of $y$ is called a {\it crosscap slide}. The mapping class $y^2$ is equal to a Dehn twist about the boundary of $K$.

We represent crosscaps as shaded disks in the picture.

\begin{figure}[ht]\label{fig:crosscap}
\begin{center}
\scalebox{0.6} 
{
\begin{pspicture}(0,-3.150067)(19.815865,3.1500673)
\pscircle[linewidth=0.04,dimen=outer,fillstyle=crosshatch*,hatchwidth=0.04,hatchangle=0,hatchsep=0.14](12.227924,0.9565343){0.3520754}
\pscircle[linewidth=0.04,dimen=outer,fillstyle=crosshatch*,hatchwidth=0.04,hatchangle=0,hatchsep=0.14](12.246314,-0.9367367){0.35368642}
\psbezier[linewidth=0.06](1.1572288,0.07721535)(1.6646175,0.31107286)(2.1212676,0.38902533)(2.1720064,0.077215314)(2.2227452,-0.2345947)(1.5377704,-0.078689694)(1.1064899,0.077215314)(0.6752094,0.23312032)(0.040973444,0.23312032)(0.09171232,0.8307562)(0.1424512,1.428392)(3.6886823,3.0800672)(3.7944503,-0.02)(3.9002182,-3.1200671)(0.066128455,-1.6462648)(0.092570454,-0.7569012)(0.11901245,0.13246232)(0.11901245,-0.12164155)(0.066128455,0.8185428)
\psbezier[linewidth=0.06,linestyle=dashed,dash=0.16cm 0.16cm](0.066128455,-0.75690114)(0.14545445,-0.07082072)(0.9122724,-0.07082077)(1.0973665,0.08164155)
\pscircle[linewidth=0.04,dimen=outer,fillstyle=crosshatch*,hatchwidth=0.04,hatchangle=0,hatchsep=0.14](2.2952008,-1.3516225){0.3439004}
\psbezier[linewidth=0.02](0.08,0.12006722)(0.0,-1.0999328)(3.02,-1.2399328)(3.06,0.20006722)(3.1,1.6400672)(0.02,0.9000672)(0.12,0.10006722)
\pscircle[linewidth=0.04,dimen=outer,fillstyle=crosshatch*,hatchwidth=0.04,hatchangle=0,hatchsep=0.14](17.833004,1.0544889){0.34699532}
\pscircle[linewidth=0.04,dimen=outer,fillstyle=crosshatch*,hatchwidth=0.04,hatchangle=0,hatchsep=0.14](17.931395,-0.7587821){0.34860635}
\psline[linewidth=0.02cm](14.14,0.020067217)(10.3,0.0)
\psbezier[linewidth=0.02](18.279072,-0.5663296)(19.721975,-0.023900408)(16.701946,0.3048445)(17.222061,1.1924559)(17.742178,2.0800672)(18.346184,1.5869498)(18.631409,1.0773951)(18.916634,0.56784046)(19.805864,-1.9799329)(18.228739,-0.99369806)
\psbezier[linewidth=0.02](17.708622,-0.53345513)(16.852947,0.008974042)(16.187382,0.38072017)(16.769056,1.1760186)(17.350733,1.971317)(17.993847,1.9650065)(18.413296,1.734885)(18.832745,1.5047636)(19.169638,-0.30333364)(19.74,-0.019932782)
\psbezier[linewidth=0.02](17.64151,-0.97726065)(16.517387,-1.7169368)(17.020727,-0.007463201)(15.9469385,-0.007463201)
\psbezier[linewidth=0.06](6.5572286,0.11721535)(7.0646176,0.35107288)(7.5212674,0.42902532)(7.572006,0.11721531)(7.622745,-0.19459471)(6.9377704,-0.038689695)(6.5064898,0.11721531)(6.0752096,0.2731203)(5.4409733,0.2731203)(5.491712,0.8707562)(5.5424514,1.468392)(9.088682,3.1200671)(9.19445,0.02)(9.300219,-3.0800672)(5.4661283,-1.6062647)(5.4925704,-0.71690124)(5.5190125,0.17246233)(5.5190125,-0.08164155)(5.4661283,0.8585428)
\psbezier[linewidth=0.06,linestyle=dashed,dash=0.16cm 0.16cm](5.4661283,-0.7169012)(5.5454545,-0.030820716)(6.3122725,-0.030820774)(6.4973664,0.12164155)
\pscircle[linewidth=0.04,dimen=outer,fillstyle=crosshatch*,hatchwidth=0.04,hatchangle=0,hatchsep=0.14](7.695201,-1.3116224){0.3439004}
\psbezier[linewidth=0.02](5.52,0.04006722)(5.4214625,-0.9550661)(7.22,-0.5799328)(7.12,-0.75993276)(7.02,-0.93993276)(7.0414543,-1.7338455)(7.7,-1.7399328)(8.358545,-1.7460201)(8.167614,-0.8962331)(7.96,-0.6799328)(7.752386,-0.46363246)(9.14,-0.07993279)(8.42,0.68006724)(7.7,1.4400672)(5.6185374,1.0352006)(5.52,0.04006722)
\pscircle[linewidth=0.06,dimen=outer](12.21,0.010067217){1.95}
\pscircle[linewidth=0.06,dimen=outer](17.85,0.07006722){1.95}
\psline[linewidth=0.08cm](3.98,0.0)(5.22,-0.019932782)
\psline[linewidth=0.08cm](5.2,-0.019932782)(4.92,0.18006721)
\psline[linewidth=0.08cm](5.18,-0.019932782)(4.94,-0.19993278)
\psline[linewidth=0.08cm](14.44,0.04006722)(15.68,0.020067217)
\psline[linewidth=0.08cm](15.66,0.020067217)(15.38,0.22006722)
\psline[linewidth=0.08cm](15.64,0.020067217)(15.4,-0.15993278)
\usefont{T1}{ptm}{m}{n}
\rput(4.5465627,0.5400672){\LARGE y}
\usefont{T1}{ptm}{m}{n}
\rput(15.046562,0.56006724){\LARGE y}
\end{pspicture} 
}
\end{center}
\caption{Crosscap Slide}
\end{figure}
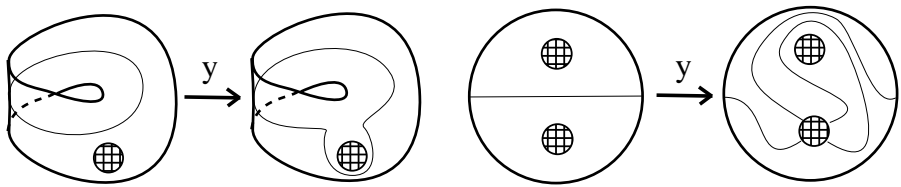

For $M$ a non-orientable surface, the mapping class group $\Gamma_M$ is generated by Dehn twists and crosscap slides (see \cite{lickorish}). Moreover, the number of generators can be chosen to be finite (see \cite{korkmaz}).






\section{The construction of an invariant measure on the moduli space}\label{section:mesure}

The aim of this section is to define a natural measure on $\X (M)$ that is invariant under the action of the mapping class group. First, we define a measure on the representation space $\Hom (\Pi , G)$ in a more general context using ideas from \cite{mulase}.

\subsection{Measure on $\Hom (\Pi, G)$}

Let $G$ be a compact Lie group. Let $$\Pi= \langle a_1 , \dots , a_k | q_1 (a_1, \dots , a_k) , \dots , q_r (a_1 , dots , a_k) \rangle$$ be a finitely presented group generated by $k$ elements with $r$ relations. We associate the {\it presentation map}
\begin{align*}
q : G^k &\longrightarrow G^r \\
x &\longmapsto (q_1 (x) , \dots , q_r (x))
\end{align*}
For $x = (x_1 , \dots , x_k)$, the element $q_j (x) = q_j (x_1 , \dots , x_k)$ of $G$ is obtained when we replace in the word $q_j (a_1 , \dots , a_k)$ the letters $a_i$ by elements $x_i$ of $G$.

There is a canonical identification between $\Hom (\pi , G)$ and the fiber $q^{-1} (1 , \dots , 1)$ of the presentation map provided by: 
\begin{equation}
p ( \Hom (\Pi , G) ) = q^{-1} (1 , \dots , 1)
\end{equation}
where the map $p$ is
\begin{align*}
p : \Hom (\Pi , G) & \longrightarrow G^k \\
\phi & \longmapsto (\phi (a_1) , \dots , \phi (a_k)).
\end{align*}
Let $dx$ be a Haar measure on $G$. The group $G$ being compact, the measure is left and right invariant. The Dirac distribution on $G$ is the linear continuous functional $\delta : C^{\infty} (G) \rightarrow \mathbb{R}$ given by:
$$ g \mapsto \int_G \delta(x) g(x) dx = g(1) , \hspace{1cm} \mbox{for any } g \in \C^{\infty} (G) $$
The Dirac distribution on $G^r$ is defined by $ \delta_r (w_1 , \dots , w_r ) = \delta(w_1) \dots \delta (w_r) $.

Let $f_q$ be the {\it volume distribution} defined as
$$f_q (w) = \int_{G^k} \delta_r (q(x) \cdot w^{-1} ) dx_1 \cdots dx_k , \hspace{1cm} w \in G^r $$
This distribution equals the linear continuous functional  $C^{\infty} (G^r) \rightarrow \R$, 
$$g \mapsto \int_{G^k} g(q ( x) )  dx_1 \cdots dx_k = \int_{G^r} f_q (w) g(w) dw_1 \cdots dw_r .$$

Distributions cannot be evaluated in a meaningful way in general. However, a distribution $f$ is said to be {\it regular} at $w \in G^r$ if there is an open neighborhood $U$ of $w$ such that the restriction of $f$ to $U$ is a $C^{\infty}$ function on $U$.

Assume that the volume distribution $f_q$ is regular at $(1, \dots , 1) \in G^r$.  Let $\mu$ be the borelian measure on $\Hom (\Pi , G)$ defined by
\begin{equation}\label{def:mesure}
\mu_q ( U ) = \int_{G^k} \delta_r (q ( x) ) {\bf 1}_{p(U)} ( x ) dx_1 \cdots dx_k
\end{equation}
for any borelian $U \subset \Hom (\Pi , G)$, where ${\bf 1}_{E}$ is the characteristic function of $E$. The total volume $\mu_q (\Hom (\Pi , G) ) = f_q (1)$ is well-defined, and hence $\mu_q$ is a finite measure on $\Hom(\Pi, G)$.

\subsection{Invariance of the measure}

The measure $\mu_q$ is defined using the presentation $q$ of the group $\Pi$. The following proposition shows that, under certain hypotheses, the measure does not depend on the choice of the presentation of $\Pi$.

\begin{proposition}\label{prop:inv}
Let $q$ and $s$ be two presentations of the same group $\Pi$
$$\Pi = \langle a_1 , \dots , a_k | q_1 , \dots , q_r \rangle  = \langle b_1 , \dots , b_l | s_1 , \dots , s_t \rangle . $$ 
Assume that $k-r = l-t$. If the volume distributions $f_q$ and $f_s$ associated to the presentation maps $q$ and $s$ are regular at $(1,\dots , 1) \in G^r$ and $(1 , \dots , 1) \in G^l$ respectively, then the measures $\mu_q$ and  $\mu_s $ coincide.
\end{proposition}
\begin{proof}
First, assume that $k=l$ and $r=t$. In this case, the proof of this proposition is deeply related to the two following lemmas, whose proofs can be found in \cite{mulase}.
\begin{lemma}\label{lem:foncteur}
Let $q$ and $s$ be two presentations of the same group $\Pi$
$$\Pi = \langle a_1 , \dots , a_k | q_1 , \dots , q_r \rangle  = \langle b_1 , \dots , b_k | s_1 , \dots , s_r \rangle . $$ 
Then for every $a_i , i= 1, \dots, k$ there is a word $a_i (b)$ in the generators $b_1 , \dots , b_k$, and for every $b_j , j=1 , \dots , k$ there is a word $b_j (a)$ in the generators $a_1 , \dots , a_k$, such that the maps $a : G^k \rightarrow G^k$ and $b : G^k \rightarrow G^k$ associated to these words are bijective and the following diagram:
$$
\xymatrix{
G^k \ar@{=}[r] & G^k \ar[r]^q \ar[d]_{\sim}^b & G^r \\
G^k  \ar[u]^a_{\sim} \ar@{=}[r]  & G^k \ar[r]^s & G^r }
$$
is commutative. Moreover, 
 $$q^{-1}(1) = b^{-1} ( s^{-1} (1) ) , \hspace{1cm} \mbox{ and } \hspace{1cm} s^{-1}(1) = a^{-1} ( q^{-1} (1) ).$$
\end{lemma}

The isomorphisms $a$ and $b$ are real analytic automorphisms of the real analytic manifold $G^k$.

\begin{lemma}\label{lem:auto}
Suppose that the map,
\begin{align*}
b : G^k & \longrightarrow G^k \\
(x_1 , \dots , x_k) = x & \longmapsto (b_1 (x) , \dots , b_k(x) )
\end{align*}
is an analytic automorphism of the real analytic manifold $G^k$ given by $k$ words $b_1 , \dots , b_k$ in $x_1 , \dots , x_k$ such that its inverse has the same form. Then the volume form $\omega^k$ of $G^k$ corresponding to the product of Haar measures $dx_1 \dots dx_k$ is invariant under the automorphism $b$ up to a sign, {\it i.e.} $ b^* \omega^k = \pm \omega^k$.
\end{lemma}

Now let the two presentation $q$ and $s$ satisfy the hypotheses of Lemma \ref{lem:foncteur}, and let $a$ and $b$ be the analytic automorphisms of $G^k$ given by this lemma. Let $p_q$ and $p_s$ be the maps $\Hom (\Pi_i , G) \rightarrow G^k$. By Lemma \ref{lem:foncteur}, we have $p_s (V) = b (p_q (V))$ and $\delta_r (q (x)) = \delta_r (s(b(x))) $ and by Lemma \ref{lem:auto} the measure $dx_1 \cdots dx_k$ on $G^k$ is invariant by the automorphism $b$. Hence, for $V$ a borelian of $\Hom (\Pi, G)$ we have
\begin{align*}
\mu_s (V) & = \int_{G^k} \delta_r (s ( x) ) {\bf 1}_{p_2 (V)} ( x ) dx_1 \cdots dx_k \\
& = \int_{G^k} \delta_r (s(b(b^{-1}( x) ))) {\bf 1}_{b(p_1(V))} ( x ) dx_1 \cdots dx_k \\
& = \int_{G^k} \delta_r (q ( b^{-1}( x)) ) {\bf 1}_{p_1(V)} ( b^{-1}(x) ) dx_1 \cdots dx_k \\
& = \mu_q (V)
\end{align*}

In the general case, we can assume without loss of generality that $k < l$. We add the new generators $a_{k+1} , \dots , a_l$ and relators $q_{r+1} =  a_{k+1} , \dots , q_{r+l-k} = a_l $ to the presentation $q$ of the group $\Pi$. The new presentation $q'$ given by
$$\langle a_1 , \dots , a_l | q_1 , \dots , q_{r+l-k} \rangle$$
is also a presentation of the group $\Pi$. 
If $f_q$ is regular at $(1 , \dots , 1) \in G^r$ then the volume distribution $f_{q'}$ of the new presentation $q '$ is also regular at $(1, \dots , 1) \in G^l$. Consider the following commutative diagram
$$
\xymatrix{
G^k \times G \ar[r]^{ (q , id ) } & G^r \times G  \\
G^k  \ar[u]^{i_k} \ar[r]_{q}  & G^r \ar[u]_{i_r} }
$$
where $i_k$ and $i_r$ are the canonical injection of $G^k$ in $G^k \times \{1\} \subset G^k \times G$. The identity 
$$\int_{G^k \times G} \delta_r (q(x)) \delta (x) dx_1 \cdots dx_k \cdot dx = \int_{G^k} \delta_r (q(x)) dx_1 \cdots dx_k ,$$
implies that the measure $\mu_q$ and $\mu_{q'}$ defined by the presentations $q$ and $q '$ coincide. The two presentations $q '$ and $s$ have the same number of generators and relations and hence the measure $\mu_{q'}$ and $\mu_s$ coincide. This ends the proof of Proposition \ref{prop:inv}.\end{proof}

The natural action of $\mbox{Aut} (\Pi) \times \mbox{Aut} (G)$ on the representation space $\Hom(\Pi, G)$ is given by :
$$ (\tau , \alpha) \cdot \rho \longmapsto \alpha \circ \rho \circ \tau^{-1}$$
for any $\rho \in \Hom(\Pi, G)$ and $(\tau, \alpha) \in \mbox{Aut}(\Pi) \times \mbox{Aut}(G)$. The group of inner automorphism of $G$, denoted $\mbox{Inn} (G)$, is the subgroup of $\mbox{Aut}(G)$ consisting of elements of the form 
$L_g : G \rightarrow G, L_g (h) = g h g^{-1}$ for all $h \in G$, with $g \in G$. We have the following proposition : 
\begin{proposition}\label{prop:inn}
Let $q$ be a presentation of $\Pi$ such that the distribution $f_q$ is regular at $(1, \dots , 1) \in G^r$. The measure $\mu_q$ on the representation space $\Hom ( \Pi , G)$ is invariant under the action of the group of inner automorphisms of $G$.
\end{proposition}
\begin{proof}
Let $g$ be an element of $G$, and $V$ be a borelian of the representation space $\Hom (\Pi, G)$. The Dirac distribution $\delta$ is invariant by conjugation, and so is the distribution $\delta_r$. The Haar measure $dx$ is also invariant by conjugation as a left and right invariant measure. Hence, we have
\begin{align*}
\mu_q( g \cdot V ) & = \int_{G^k} \delta_r (q ( x) ) {\bf 1}_{p(g\cdot V)} ( x ) dx_1 \cdots dx_k \\
& = \int_{G^k} \delta_r (g^{-1} q ( x) g ) {\bf 1}_{g p(V) g^{-1}} ( x ) dx_1 \cdots dx_k \\
& = \int_{G^k} \delta_r ( q ( g^{-1}x g )) {\bf 1}_{ p(V) } (g^{-1} x g) dx_1 \cdots dx_k \\
& = \mu_q (V)
\end{align*}
\end{proof}

An automorphism of $\Pi = \langle a_1 , \dots , a_k | q_1 , \dots , q_r \rangle$ is given by $k$ words in $a_1 , \dots , a_k$, and its inverse is of the same form. Hence, we have the immediate corollary to Proposition \ref{prop:inv}
\begin{cor}\label{cor:aut}
Let $q$ be a presentation of $\Pi$ such that the distribution $f_q$ is regular at $(1, \dots , 1) \in G^r$. The measure $\mu_q$ on $\Hom (\Pi , G) $ is invariant under the action of ${\rm Aut} (\Pi)$.
\end{cor}

\subsection{Regularity of volume distributions for non-orientable surface groups in $\SU (2)$}\
For a closed non-orientable surface $M$ of genus $k$, we take the usual presentation:
$$ \pi_1 (M) = \langle a_1 , \dots , a_k | a_1^2 \cdots a_k^2 \rangle. $$
The volume distribution $f_k$ becomes
$$f_k (w) = \int_{G^k} \delta_r (x_1^2 \cdots x_k^2 \cdot w^{-1}) dx_1 \cdots dx_k .$$
To show that this distribution is regular at the identity element, we compute its character expansion. We first have to set some notations.

Let $\widehat{G}$ denote the set of isomorphism classes of complex irreducible representations of $G$ and let $\chi_{\lambda}$ be the character of the irreducible representation $\lambda \in \widehat{G}$. Using the Frobenius-Schur indicator of irreducible characters (see \cite{frobenius}), we decompose $\widehat{G}$ into the disjoint union of the three following subsets:
\begin{eqnarray*}
\widehat{G}_1 & = & \left\{ \lambda \in \widehat{G} \, \bigg| \, \frac{1}{|G|} \int_G \chi_{\lambda} (w^2) dw = 1 \right\}, \\ 
\widehat{G}_2 & = & \left\{ \lambda \in \widehat{G} \, \bigg| \, \frac{1}{|G|} \int_G \chi_{\lambda} (w^2) dw = 0 \right\}, \\ 
\widehat{G}_4 & = & \left\{ \lambda \in \widehat{G} \, \bigg| \, \frac{1}{|G|} \int_G \chi_{\lambda} (w^2) dw = -1 \right\} .
\end{eqnarray*}

With these notations we state the following proposition, whose detailed proof can be found in \cite{mulase}.
\begin{proposition}\label{prop:character}
The character expansion of the volume distribution $f_k$ is given by
\begin{equation}\label{eqn:exp}
f_k (w) = \sum_{\lambda \in \widehat{G}_1} \left( \frac{ |G| }{\dim \lambda} \right)^{k-1} \chi_{\lambda} (w) - \sum_{\lambda \in \widehat{G}_4} \left( - \frac{ |G| }{\dim \lambda} \right)^{k-1} \chi_{\lambda} (w) 
\end{equation}

If the right-hand side sum is absolutely convergent for $w = 1$, then it is uniformly and absolutely convergent on $G$, and the volume distribution $f_k$ is a $C^{\infty}$ function.

When $G = \SU (2)$, the series associated to $f_k (1)$ are absolutely convergent for $k\ge 4$.
\end{proposition}

\begin{proof}The proof of Proposition \ref{prop:character} relies on the convolution property of the $\delta$-function. Namely, let $q_k$ be the word in $a_1 , \dots , a_k$ given by $q_k (a_1 , \dots , a_k) = a_1^2 \cdots a_k^2$. We have that $$\displaystyle{\delta (q_{g+h} w^{-1} )= \int_G \delta( q_h w^{-1} u^{-1}) \delta ( u q_g) du   }$$
and hence $$\displaystyle{ f_k =  \stackrel{k-{\rm times}}{ \overbrace{f_1 \ast \cdots \ast f_1}} }.$$
Here $f \ast g$ denotes the usual convolution product, given by 
$$(f \ast g) (x) = \int_G f(x w^{-1}) g(w) dw .$$

The irreducible characters are real analytic functions on $G$ and form a orthonormal basis for the $L^2$ class function on $G$. The character expansion in terms of irreducible characters of the class distribution $f_1$ is given by:
$$f_1 (w) = \sum_{\lambda \in \widehat{G}} Z_{\lambda} \chi_{\lambda} (w)$$
where
$$Z_{\lambda} = \frac{1}{|G|} \int_G \chi_{\lambda} (x^2) dx.$$

Hence using the decomposition $\widehat{G}= \widehat{G}_1 \sqcup \widehat{G}_2  \sqcup \widehat{G}_4$ given by the Frobenius-Schur indicator, we obtain
$$f_1 (w) = \sum_{\lambda \in \widehat{G}_1} \chi_{\lambda} (w) - \sum_{\lambda \in \widehat{G}_4} \chi_{\lambda} (w)$$

The convolution property of irreducible characters states that 
$$\chi_\lambda \ast \chi_\mu = \frac{|G|}{\dim \lambda} \delta_{\lambda \mu} \chi_{\lambda}.$$
This formula applied $k-1$ times to the convolution $f_k = f_1 \ast \cdots \ast f_1$ gives us the formula (\ref{eqn:exp}).

Moreover, for $G = \SU (2)$, we know that the dimension of the representation in $\widehat{G}_1$ consists of odd integers and in $\widehat{G}_4$ of even integers. Hence 
$$f_k (1) = |G|^{k-1} \left( \sum_{n=1}^{\infty} (2n - 1)^{2-k} + (-1)^{2-k} \sum_{n=1}^{\infty} (2n)^{2-k} \right)$$
which is absolutely convergent for $k\geq 4$.
\end{proof}

Let $M$ be a closed non-orientable surface of genus $k \geq 4$. We have defined a measure $\mu$ on $\Hom(\pi, \SU (2))$. Consider the quotient map
\begin{equation*}
Q : \Hom (\pi , \SU(2) )  \longrightarrow \Hom (\pi , \SU(2)) / \SU(2) =  \X(M)
\end{equation*}
and define a measure $\nu$ on $\X(M)$ as the push-forward measure of the measure $\mu$ through $Q$, given by:
\begin{equation}
\nu (V) = \mu (Q^{-1} (V) )
\end{equation}
Then Corollary \ref{cor:aut} together with Proposition \ref{prop:inn} show that the measure $\nu$ is $\mbox{Out} (\Pi)$-invariant on the quotient $\Hom (\Pi , \SU (2) ) / \SU(2)$. Moreover, this measure is independent on the choice of the presentation of $\pi$. The ergodicity result of Theorem \ref{closed} will be proved with respect to this measure.

\begin{remark} The fundamental group of a surface $M$ with boundary is a free group on $l = 1 - \chi(M)$ generators. The representation space $\Hom (\pi , G)$ is isomorphic to $G^l$ with the natural presentation of the free group with $l$ generators and no relation. Hence, the measure obtained by the above construction is simply the Haar measure on $G^l$.
\end{remark}

\section{Surfaces decompositions and Goldman's flow}\label{section:flot}

\subsection{Goldman's flow}

Let $f :  G \rightarrow \R$ be a $C^1$ function invariant under inner automorphisms of $G$, namely satisfying $f ( PAP^{-1}) = f(A)$, for all $A,P \in G$. For the rest of this paper $G$ will denote the group $\SU (2)$ and we will consider henceforth $f (A) = \cos^{-1} \left( \frac{\tr (A)}{2} \right) \in [ 0 , \pi ]$, where $\tr$ denotes the usual trace in $\SU(2)$.

Let $\mathfrak{g}$ denote the Lie algebra of $G$ and let $\langle X , Y \rangle$ denote the inner product on $\mathfrak{g}$ (i.e. the Killing form) defined by $< X , Y > = - \Tr (XY)$, for all $X,Y$ in $\mathfrak{g}$. The {\it variation} of the function $f$ is the $G$-equivariant function $F : G \rightarrow \mathfrak{g}$ 
 defined by the equation:
$$<X , F(A)> = df_A (X) = \frac{d}{dt} f (A \exp (tX) ) , \hspace{1cm} \mbox{for any } A \in G, \mbox{ and } X \in \mathfrak{g}.$$

If $A$ is a matrix of the form $\displaystyle{ \left( \begin{array}{cc} e^{i \theta} & 0 \\ 0 & e^{-i \theta} \end{array} \right) }$ with $\theta \in ] 0 , \pi [ $, then $ \displaystyle{F(A) = \left( \begin{array}{cc} -i & 0 \\ 0 & i \end{array} \right) }$. For $B = g A g^{-1}$ with $g \in G$ and $A$ of the above form, we have 
\begin{equation}\label{variation} F(B) = F(gAg^{-1})= g \left( \begin{array}{cc} -i & 0 \\ 0 & i \end{array} \right) g^{-1} .\end{equation}
The function $F$ is defined on $G \setminus \{ \pm I \}$, but is not defined for the extremal values $\theta \in \{ 0 , \pi \}$. However, for our purpose it will be sufficient to define the flow on an open dense subset of full-measure of the character variety. 

Let $M$ be a compact non-orientable surface and $\gamma$ a two-sided circle on $M$. Let $f_{\gamma} :  \Hom (\pi_1(M) , G ) / G  \longrightarrow \R $ be the function defined on $\X (M) $ by:
$$ f_{\gamma} ([ \rho ]) =   f ( \rho (\gamma)).$$
We define the open dense subset $\widetilde{\mathcal{S}_{\gamma}}$ of full measure of $\Hom(\pi , G)$ as:
 $$\widetilde{\mathcal{S}_{\gamma}} = \{ \rho \in \Hom (\pi , G) | \rho(\gamma) \neq \pm I \}.$$
Let $\mathcal{S}_{\gamma}$ denote its image in $\X(M)$. For $\rho \in \widetilde{\mathcal{S}_{\gamma}}$, let $\zeta_t (\rho) = \exp ( t F(\rho (\gamma)))$. This defines a path in the centralizer $Z(\rho(\gamma))$ of $\rho (\gamma)$ in $G$.

We construct a flow on $\mathcal{S}_{\gamma}$ called a {\it generalized twist flow} or {\it Goldman flow} (see \cite{jefwei,klein}). We define the flow $\Xi_t (\rho)$ in the two cases of interest for us, namely when $\gamma$ is a separating circle, and when $\gamma$ is a non-separating circle such that $M|\gamma$ is orientable. The other situation,  corresponding to a non-separating circle such that $M|\gamma$ is non-orientable, will not be used in the sequel but can be treated in the same way than the case when $M | \gamma$ is orientable.\\

\subsection{The flow associated to a separating circle}

Let $\gamma$ be a separating circle on $M$. Then $M|\gamma$ is the disjoint union of two subsurfaces $A$ and $B$. Without loss of generality, we can assume that $A$ is non-orientable. We place a base point $p$ on the circle $\gamma$. The surface $M$ is obtained by gluing $A$ and $B$ along the circle $\gamma$. Hence, the Seifert-Van Kampen theorem shows that the fundamental group $\pi_1 (M)$ can be reconstructed from $\pi_1 (A)$ and $\pi_1 (B)$ as
$$ \pi_1 (M, p) = \pi_1 (A,p ) \ast_{\pi_1 (\gamma , p )} \pi_1 (B,p) .$$
The fundamental group $\pi_1 (\gamma, p)$ is isomorphic to the cyclic group $\mathbb{Z}$. We also denote by $\gamma$ the class of the curve $\gamma$ in $\pi_1 (M)$. Hence we have $\pi_1 (\gamma) = \langle \gamma \rangle$.

The flow on $\widetilde{\mathcal{S}_{\gamma}}$ is defined by:
\begin{equation}\label{flot1}
\widetilde{\Xi}_t \rho (\delta ) = \left\{ \begin{array}{ll} \rho (\delta ), & \mbox{ if } \delta  \in \pi_1 (A), \\ \zeta_t (\rho) \rho (\delta ) \zeta_t (\rho)^{-1}, & \mbox{ if } \delta  \in \pi_1 (B) \end{array} \right.
\end{equation}
where $\rho$ is an element of $\Hom (\pi_1(M) , G)$, and $t$ is a real number.
The element $\zeta_t (\rho)$ is in the centralizer of $\rho (\gamma)$, hence $\rho (\gamma) = \zeta_t (\rho) \rho (\gamma) \zeta_t (\rho)^{-1}$, and the element $\widetilde{\Xi}_t \rho (\gamma)$ is well-defined.  

We define the flow  $\{ \Xi_t \}_{t\in \R}$ on $\mathcal{S}_{\gamma}$, such that it is covered by $\{ \widetilde{\Xi}_t \}_{t \in \R}$. For any representation $\rho$ in $\widetilde{\mathcal{S}_{\gamma}}$, the formula (\ref{variation}) gives 
$$\zeta_{\pi} (\rho) = \exp ( \pi F(\rho (\gamma))) = \exp \left( \pi g \left( \begin{array}{cc} -i & 0 \\ 0 & i \end{array} \right) g^{-1} \right) = -I.$$ Similarly $\zeta_{2\pi }(\rho) = I$. So we have $\widetilde{\Xi}_{\pi} \rho = \rho$, and thus the flow $\{ \Xi_t \}$ is $\pi$-periodic and defines a circle action on an open dense subset of full measure of $\X (M)$.\\

We then define $\X (M; A,B,\gamma) $ as the pull-back in the diagram:
$$ \begin{array}{clc}
\X (\gamma) & \longleftarrow & \X(A) \\
\uparrow & & \uparrow \\
\X(B) & \longleftarrow & \X (M; A,B,\gamma) 
\end{array} $$
Namely, $\X (M; A,B,\gamma) $ is the set of pairs $ ( [ \alpha ] , [\beta ] ) \in \X(A) \times \X(B) $ such that
$$[\alpha _{\mid_{\pi_1 (\gamma)} } ] = [\beta _{\mid_{\pi_1 (\gamma)} } ] \in \X(\gamma) $$
We have a natural map $j : \X (M)  \longrightarrow \X(M; A,B,\gamma) $ given by:
 $$j([ \rho ]) = ([\rho _{\mid _{\pi_1 (A)}} ] , [ \rho _{\mid _{\pi_1 (B)}} ]).$$

\begin{proposition} The generic fibers of the map $$j \, : \, \X(M) \longrightarrow \X (M;A,B,\gamma)$$ are the orbits of the circle action $\{ \Xi_t \}$. \end{proposition}

\begin{proof}
Let $( [ \alpha ] , [\beta ] )$ be an element of $\X (M; A,B,\gamma)$ and let $\alpha$ and  $\beta$ be representatives of $[\alpha]$ and $[\beta]$ respectively. By definition, there exists an element $g$ of $G$ such that $\alpha_{\mid_{\pi_1 (\gamma)} } = g \cdot \beta_{\mid_{\pi_1 (\gamma) }} \cdot g^{-1}$. Without loss of generality, we can choose a representative $\beta$ such that $\alpha_{\mid_{\pi_1 (\gamma)} } = \beta_{\mid_{\pi_1 (\gamma) }}$. Let $[\rho]$ be a conjugacy class of representation in $\X (M) $ and let $\rho$ be a representative of $[\rho]$. The class $[\rho]$ is in the fiber $j^{-1}( [ \alpha ] , [\beta ] )$ if and only if there exists $h_1, h_2$ in $G$ such that $\rho _{\mid _{\pi_1 (A)}} = h_1 \alpha h_1^{-1}$ and $\rho _{\mid _{\pi_1 (B)}} = h_2 \beta h_2^{-1}$. Without loss of generality, we can choose a representative $\rho$ of $[ \rho ]$ such that $h_1 = 1$. Then we obtain $ \alpha(\gamma) = h_2 \beta (\gamma) h_2^{-1} =h_2  \alpha (\gamma) h_2^{-1}$. It follows that $h_2$ is in the centralizer $Z(\alpha (\gamma))$ of $\alpha(\gamma)$. Henceforth the fiber is identified with the centralizer $Z(\alpha (\gamma))$.

If $\alpha(\gamma) \neq \pm Id$ then $Z(\alpha (\gamma) )$ is a maximal torus in $SU(2)$ which acts simply transitively on itself by left multiplication. Therefore, we identify the maximal torus $Z(\alpha(\gamma))$ with the space $\{ \zeta_t (\rho ) | t \in [ 0 , 2 \pi ] \}$ when $\alpha(\gamma) \neq \pm I$. Finally, for a generic element of $\X (M; A,B,\gamma)$, we have $\alpha(\gamma) \neq \pm Id$, which ends the proof of the proposition.

\end{proof}

\subsubsection*{Relation with the measure}
We can choose a presentation of $M$ as $$\pi_1 (M) = \langle A_1 , \dots , A_k , B_1 , \dots B_l | q_k (A_1 , \dots , A_k) ( q_l (B_1 , \dots , B_l) )^{-1} \rangle $$ where $q_n$ is the word defined for any $n$ in $\mathbb{N}$ by $ q_n (x_1 , \dots , x_n) = x_1^2 \cdots x_n^2$. The fundamental group of $A$ and $B$ are given by;
\begin{align*}
\pi_1 (A) &= \langle A_1 , \dots , A_k , C | q_k (A_1 , \dots , A_k) C^{-1} \rangle\\
\pi_1 (B) &= \langle B_1 , \dots , B_l , C | q_l (B_1 , \dots , B_l) C^{-1} \rangle.
\end{align*}

The circle $\gamma$ on $M$ is represented by $q (A_1 , \dots , A_k)$. In this setting, the flow on $\Hom (\pi , G)$ acts by left and right multiplication on the $k$ first generators, and hence is measure preserving. 

Moreover, we can see the fundamental group of the surface with boundary $A$ as a free group with generators $A_1 , \dots , A_k$. The restriction map $\X(M) \longrightarrow \X(A)$ is defined by the image of these generators. Hence, the measure on $\X(A)$ defined as the push-forward measure through this restriction, is in the class of the Haar measure on $G^k / G$. And the same result holds with the restriction $\X(M) \longrightarrow \X (B)$.

Finally the decomposition measure $\nu_{[\alpha],[\beta]}$, with respect to the map $j$, on the fiber $j^{-1}  ([\alpha],[\beta])$ is the Haar measure on the maximal torus $Z ( \alpha (\gamma))$, which is a circle. Thus the Haar measure is in the Lebesgue class on $\mathbb{S}^1$. 

\subsubsection*{The action of the Dehn twist}
With the identification $ \pi_1 (M) = \pi_1 (A ) \ast_{\pi_1 (\gamma)} \pi_1 (B)$, the Dehn twist $\tau_{\gamma}$ about the curve $\gamma$ acts on an element $\rho$ in $\Hom (\pi_1(M) , G)$ as
$$ (\tau_{\gamma} \cdot \rho )( \delta ) = \left\{ \begin{array}{ll}  \rho (\delta) , & \mbox{ if } \delta \in \pi_1 (A), \\ \rho(\gamma) \cdot \rho(\delta) \cdot \rho (\gamma)^{-1}, & \mbox{ if } \delta \in \pi_1 (B) \end{array} \right. $$
For any $g \in G$ we have $g = \exp (f (g) \cdot F (g) )$, so we obtain 
$$\rho (\gamma) = \exp (f(\rho(\gamma)) \cdot F (\rho (\gamma))) = \zeta_{f(\rho(\gamma))} (\rho).$$ 
Therefore, the Dehn twist on $\X(M)$ can be expressed in terms of the Goldman flow :
\begin{equation}\label{dehn1}
\tau_{\gamma} = \Xi_{f ( \rho ( \gamma))}.
\end{equation}
The Goldman flow is a $\pi$-periodic circle action on a generic fiber of the application $j$, which is homeomorphic to a circle. So the twist $\tau_{\gamma}$  acts on the fiber $j^{-1}([ \alpha ] , [\beta ] )$ of the application $j : \X (M) \rightarrow \X (M ; A, B , \gamma)$ as the rotation of angle $2 f(\rho ( \gamma))$ on this circle with $[\rho] \in j^{-1}([ \alpha ] , [\beta ] )$.

\subsection{The flow associated to a non-separating circle}

Let $M = N_{2g + 2 , m}$ be the orientable surface of genus $g$ with $m$ boundary components, with two crosscaps attached. The circle $\gamma$ is a two-sided curve passing through the two crosscaps (see Figure \ref{fig:non-sep}). The surface $A = M|\gamma$ is an orientable surface of genus $g$ with $m+2$ boundary components. The two additional boundary components that correspond to the two sides of $\gamma$ are denoted $\gamma_+$ and $\gamma_-$. Recall that crosscaps are drawn as shaded disks.

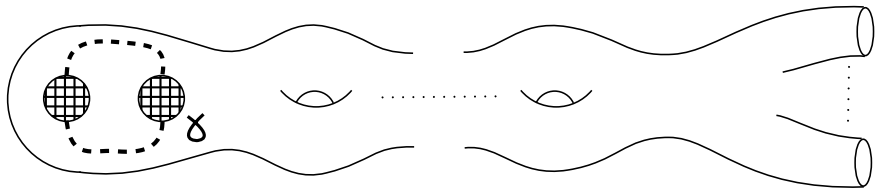
\begin{figure}[ht]\label{fig:non-sep}
\begin{center}
\scalebox{0.7} 
{
\begin{pspicture}(0,-1.805)(16.46,1.805)
\psarc[linewidth=0.03](1.39,-0.04){1.39}{90.0}{270.0}
\pscircle[linewidth=0.03,dimen=outer,fillstyle=crosshatch*,hatchwidth=0.04,hatchangle=0,hatchsep=0.14](1.12,-0.03){0.45}
\pscircle[linewidth=0.03,dimen=outer,fillstyle=crosshatch*,hatchwidth=0.04,hatchangle=0,hatchsep=0.14](2.92,-0.03){0.45}
\psarc[linewidth=0.03](5.83,-0.28){0.39}{25.0}{155.0}
\psarc[linewidth=0.03](5.86,0.69){0.88}{-140.0}{-40.0}
\psarc[linewidth=0.03](10.39,-0.28){0.39}{25.0}{155.0}
\psarc[linewidth=0.03](10.42,0.69){0.88}{-140.0}{-40.0}
\psbezier[linewidth=0.03](1.36,1.35)(2.3140507,1.4310776)(2.8705063,1.2170984)(3.8443038,0.9235492)(4.8181014,0.63)(5.1546054,1.5401132)(6.0502534,1.334518)(6.945901,1.1289228)(6.8653164,0.8461295)(7.7,0.83)
\psbezier[linewidth=0.03](8.66,0.85)(9.36,0.83)(9.660341,1.496122)(10.64,1.33)(11.619658,1.1638781)(11.743482,0.7533701)(12.64,0.81)(13.536519,0.8666299)(14.24,1.79)(16.26,1.71)
\psbezier[linewidth=0.03](1.36,-1.43)(2.32,-1.55)(2.92,-1.31)(3.86,-1.05)(4.8,-0.79)(5.14,-1.67)(6.0502534,-1.445482)(6.9605064,-1.220964)(6.9053164,-0.9338705)(7.72,-0.95)
\psbezier[linewidth=0.03](8.68,-0.97)(9.36,-0.89)(9.680341,-1.5438781)(10.64,-1.39)(11.599659,-1.236122)(11.703482,-0.7666299)(12.56,-0.77)(13.416518,-0.7733701)(14.0,-1.79)(16.26,-1.71)
\psline[linewidth=0.04cm,linestyle=dotted,dotsep=0.16cm](7.1,-0.01)(9.28,0.01)
\psellipse[linewidth=0.03,dimen=outer](16.29,1.24)(0.17,0.47)
\psellipse[linewidth=0.03,dimen=outer](16.25,-1.25)(0.17,0.46)
\psbezier[linewidth=0.03](16.28,0.77)(15.8,0.83)(15.26,0.59)(14.72,0.47)
\psbezier[linewidth=0.03](16.22,-0.79)(15.36,-0.75)(14.98,-0.41)(14.6,-0.35)
\psline[linewidth=0.04cm,linestyle=dotted,dotsep=0.16cm](15.98,0.59)(15.96,-0.47)
\psbezier[linewidth=0.08,linestyle=dashed,dash=0.16cm 0.16cm](1.12,0.41)(1.14,0.71)(1.16,1.11)(1.98,1.05)(2.8,0.99)(3.02,0.95)(2.94,0.43)
\psbezier[linewidth=0.08,linestyle=dashed,dash=0.16cm 0.16cm](2.94,-0.47)(2.94,-1.07)(2.3955636,-1.0640916)(2.02,-1.03)(1.6444365,-0.9959085)(1.2,-1.23)(1.12,-0.45)
\psbezier[linewidth=0.06](3.42,-0.37)(4.0,-0.81)(3.6,-0.83)(3.6,-0.83)(3.6,-0.83)(3.16,-0.85)(3.72,-0.33)
\end{pspicture} 
}
\end{center}
\caption{A non orientable surface of even genus}
\end{figure} 

The surface $M$ is obtained from $A$ by gluing the two boundary components $\gamma_+ , \gamma_-$ with an orientation-reversing homeomorphism. So $\pi_1 (M)$ can be constructed from $\pi_1 (A)$ by an HNN construction. The group $\pi_1 (M)$ is the quotient of the free product of $\pi_1(A)$ with a cyclic group $\langle \beta \rangle \cong \mathbb{Z}$, by the normal subgroup generated by the set
$$ N = \left\{ i_- (\tau) \cdot \beta \cdot i_+ (\tau) \cdot \beta^{-1} \, \mid \, \tau \in \pi_1(\gamma) \right\}, $$
here $i_{\pm}$ are the embeddings induced by inclusion $\gamma \hookrightarrow \gamma_{\pm} \hookrightarrow M $.
$$ i_{\pm} \, : \, \pi_1(\gamma) \longrightarrow \pi_1(\gamma_{\pm} ) \stackrel{i_{\gamma}}{\longrightarrow}  \pi_1 (M). $$
Namely, we obtain
$$\pi_1 (M)  = (\pi_1 (A) \ast \langle \beta \rangle ) \, \slash \, N ,$$ 

The new generator $\beta$ corresponds to a one-sided circle on $M$ which crosses $\gamma$ exactly once. \\

The flow on $\widetilde{\mathcal{S}_{\gamma}}$ is defined by:
\begin{equation}\label{flot2}
\widetilde{\Xi}_t \rho  (\delta)  = \left\{ \begin{array}{ll} \rho (\delta ), & \mbox{ if } \delta  \in \pi_1 (A), \\ \zeta_t (\rho) \rho (\delta ), & \mbox{ if } \delta  \in \pi_1 (\beta) . \end{array} \right. ,
\end{equation}
where $\rho$ is an element of $\Hom (\pi_1(M) , G)$, and $t$ is a real number.

We define a flow  $\{ \Xi_t \}_{t\in \R}$ on $\mathcal{S}_{\gamma}$, that is covered by $\{ \widetilde{\Xi}_t \}_{t \in \R}$.  For any representation $\rho$ in $\widetilde{\mathcal{S}_{\gamma}}$, the formula (\ref{variation}) gives $\zeta_{2\pi }(\rho) = I$. So we have $\widetilde{\Xi}_{2 \pi} \rho = \rho$, and thus the flow $\{ \Xi_t \}$ is $2 \pi$-periodic and defines a circle action on an open dense subset of full measure of $\X (M)$.

We have a natural map $\phi \, : \, \X(M) \longrightarrow \X (A)$ given by
$$\phi ([\rho] ) = [\rho _{\mid _{\pi_1 (A)}} ]$$

\begin{proposition} The generic fibers of the map $\phi$ are the orbits of the circle action $\{ \Xi_t \}$. \end{proposition}

\begin{proof}
We denote by $\beta_0$ the preimage of $\beta$ in $A$, which is an arc with one endpoint on $\gamma_+$ and one endpoint on $\gamma_-$. Let $x_0$ be the endpoint of $\beta_0$ on $\gamma_-$. For convenience we also denote by $\gamma_-$ the corresponding element of $\pi_1 ( A ,  x_0)$. Let $\gamma_+$ be the element of $\pi_1(A,x_0)$ corresponding to the loop $\beta_0^{-1} \star \widetilde{\gamma}_+ \star \beta_0$ where $\widetilde{\gamma}_+$ is the loop following $\gamma_+$ based on the endpoint of $\beta_0$ on $\gamma_+$. A representation $\rho_A : \pi_1(A) \rightarrow G$ extends to a representation of $\pi_1(M)$ if and only if there exists $b \in G$ such that $\rho_A (\gamma_-)^{-1}$ is conjugate to $ \rho_A (\gamma_+) $, i.e. 
\begin{equation}\label{recol}
 \rho_A (\gamma_-) ^{-1} = b \rho_A (\gamma_+) b^{-1}.
\end{equation}

The choice of the element $b$ corresponds to the choice of the image of the new generator $\beta$. Two elements of $G = \SU (2)$ are conjugate if and only if they have the same trace, and an element of $G$ and its inverse have the same trace. We infer that 
$$\phi ( \X (M)) = \{ [ \rho_A ] \in \X (A) \mid \tr ( \rho_A ( \gamma_-)) = \tr (\rho_A (\gamma_+)) \} .$$
Let $[\rho_A]$ be an element of $\phi ( \X (M))$, and $\rho_A \in \Hom (\pi , G)$ a representative. Let $g$ be the element of $G$ such that $ \rho_A (\gamma_-) = g \rho_A (\gamma_+) g^{-1}$. Then the fiber $\phi^{-1} ([ \rho_A ] )$ is identified with the set of those $b \in G$ satisfying (\ref{recol}). Thus we can see that 
$$\phi^{-1} ( [ \rho_A ] ) = \left\{ b \cdot g \mid b \in Z (\rho_A (\gamma_- )) \right\} . $$
So the fiber $\phi^{-1} ([ \rho_A ] )$ is a right-coset of the centralizer $Z (\rho_A (\gamma_-))$ of $\rho_A (\gamma_- )$. If $\rho_A (\gamma_-) \neq \pm Id$ then $Z(\rho_A (\gamma_-) )$ is a maximal torus in $SU(2)$ which acts transitively on the fiber by left multiplication. Therefore, we identify the maximal torus $Z(\rho_A(\gamma_-))$ with the space $\{ \zeta_t (\rho_A (\gamma_-) ) | t \in [ 0 , 2 \pi ] \}$ when $\rho_A(\gamma_-) \neq \pm I$. Finally, the set of all $[\rho_A] \in  \phi ( \X (M))$ such that $\rho_A (\gamma) \neq \pm Id$ is an open dense subset of full measure of $\phi ( \X (M))$, which ends the proof of the proposition.
\end{proof}

\subsubsection*{Relation with the measure}
We can choose a presentation of $\pi_1 (M)$ as
$$\pi_1 (M) = \langle A_1 , \dots , A_k , \gamma , \beta | q_k (A_1, \dots , A_k) \beta \gamma \beta^{-1} \gamma \rangle $$
such that the elements $A_1 , \dots , A_k,  \gamma $ generates $\pi_1 (A)$ as a free group with $k+1$ generators. In this setting, the flow on $\Hom (\pi , G)$ acts on a representation $\rho$ by left multiplication on $\rho (\gamma)$ and let $\rho (A_1) , \dots ,\rho (A_k)$ invariant. Hence the flow is measure preserving. 

Moreover, the map $\phi : \X(M) \longrightarrow \X(A)$ is defined by the image of the first $k+1$ generators of $\pi_1 (M)$. Hence the push-forward of the measure $\nu$ on $\X(A)$ equals the quotient of Haar measure on $G^{k+1} / G$. Finally, the decomposition measure $\nu_{[\rho_A]}$, with respect to the map $\phi$, on the fiber $\phi^{-1}  ([\rho_A])$ is the Haar measure on the maximal torus $Z ( \rho_A (\gamma))$, and hence is in the Lebesgue class.

\subsubsection*{The action of the Dehn twist}
With the identification of $\pi_1 (M)  = (\pi_1 (A) \ast \langle \beta \rangle ) \, \slash \, N $, the Dehn twist $\tau_{\gamma}$ about the curve $\gamma$ acts on $\Hom (\pi_1(M) , G)$ as
$$ (\tau_{\gamma} \cdot \rho) (\delta) = \left\{ \begin{array}{ll}  \rho (\delta), & \mbox{ if } \delta \in \pi_1 (A), \\ \rho(\gamma) \cdot \rho(\delta), & \mbox{ if } \delta \in \langle \beta \rangle \end{array} . \right. $$
For any $\rho \in \Hom (\pi , G)$, we have $\rho (\gamma) = \exp (f(\rho(\gamma)) \cdot F (\rho (\gamma))) = \zeta_{f(\rho(\gamma))} (\rho)$. Then as in the previous case we express the Dehn twist in the form
\begin{equation}\label{dehn2}
\tau_{\gamma} = \Xi_{f ( \rho ( \gamma))}.
\end{equation}
The Goldman flow is a $2 \pi$-periodic circle action on a generic fiber which is a circle. So the twist $\tau_{\gamma}$ acts on a generic fiber $\phi^{-1} ( [ \rho_A ] )$ of the application $\phi : \X (M) \rightarrow \X (A)$ as the rotation of angle $f(\rho_A ( \gamma))=f(\rho (\gamma))$ on this circle, with $\rho \in \phi^{-1} ( [ \rho_A ] )$.\\

\begin{remark}
When the surface $M$ is oriented and compact, the flow defined by (\ref{flot1}) or (\ref{flot2}) covers the flow of the Hamiltonian vector field on $\X(M)$ associated to the function $f_{\alpha}$ with respect to the natural symplectic structure on the space $\X (M)$ (see \cite{Goldman86}). 
\end{remark}

\section{Surfaces of even genus}\label{section:pair}

In this section, we prove theorem $\ref{thm:open}$ in the case of a non-orientable surface of even genus. Let $M$ be the non-orientable surface $N_{2g+2 , m} $ and let $X$ be a non-separating curve such that the surface $A = M|X$ is orientable, as the curve $\gamma$ in Figure \ref{fig:non-sep}.

\subsection{Action of the Dehn twist about $X$}

According to (\ref{dehn2}) the Dehn twist $\tau_X$ about the curve $X$ acts on a generic fiber of the application $\phi : \X (M) \rightarrow \X (A)$ as the rotation of angle $f ( \rho (X))$. Let $\X _{\mathbb{Q}} (M)$ be the set of representations $[\rho]$ in $\X(M)$ such that $f(\rho(X))$ is a rational multiple of $\pi$. Then $\X_{\mathbb{Q}}(M)$ has zero measure. Specifically, we have
$$ \X_{\mathbb{Q}}(M) = \bigcup_{q \in \mathbb{Q}} \tr_X^{-1} (2 \cos (q \pi)).$$
where $\tr_X$ is the function $\tr_X ([\rho])  \mapsto \tr (\rho (X))$. The trace function is a non-constant algebraic function on $\X(M)$ which is an irreducible algebraic variety. Hence the set $\X_{\mathbb{Q}}(M)$ is a countable union of lower-dimensional subvarieties (which have zero measure). So on  the full-measure subset $\X'(M)$ defined as $\X (M) \setminus \X_{\mathbb{Q}}(M)$, the angle $f(\rho (X))$ is irrational. A rotation by an irrational angle on the circle is ergodic with respect to its Lebesgue measure. So we have a measurable map $\phi : \X(M) \longrightarrow \X(A)$ such that $\phi$ is $\tau_X$-invariant. Moreover the action of $\tau_X$ on the fiber $\phi^{-1}([\alpha])$ is  ergodic, with respect to the decomposition measure $\nu_{[\rho_A]}$, for almost all $[\alpha] \in \X(A)$.

We recall the following classical result of measure theory and refer to (\cite{furst} Theorem 5.8) for a proof:

\begin{lemma}\label{ergodic} {\bf of ergodic decomposition:}\\
Let $(X, \mathfrak{B} , \mu ) $ a measured space, $Y , Z $  Borel spaces, and $ F : X \longrightarrow Y$ a measurable map. Suppose that $\Gamma$ is a group of automorphisms of $(X, \mathfrak{B} , \mu ) $ such that $F$ is $\Gamma$-invariant. Let $\mu_y$ be the measures on $F^{-1}(y)$ obtained by disintegrating $\mu$ over $F$. Let $h : X \longrightarrow Z$ be a measurable $\Gamma$-invariant function.\

Suppose that the action of $\Gamma$ is ergodic on the fiber $(F^{-1}(y) , \mu_y )$ for almost all $y \in Y$.\

Then there exists a measurable function $H : Y \longrightarrow Z $ such that $h = H \circ F $ almost everywhere.
\end{lemma}

Let $h : \X(M) \longrightarrow \mathbb{R}$ be a $\tau_X$-invariant measurable function. By the Lemma of ergodic decomposition, there exists a function $H : \X (A) \longrightarrow \mathbb{R}$ such that $h = H \circ \phi$ almost everywhere. So a $\Gamma_M$-invariant function is almost everywhere equal to a function depending only on $\X (A)$.

\subsection{The ergodicity of the mapping class group action on $\X(A)$ }

The surface $A$ is an orientable surface with boundary. Let $g_A$ be an element of the mapping class group $\Gamma_A$ of $A$. The mapping class $g_A$ is an element of $\mbox{Out} (\pi_1 (A))$. Using the identity $\pi_1 (M) = \pi_1 (A) \ast \langle \tau \rangle / N$, we define an element $g$ of the mapping class group $\Gamma_M$ of $M$. First, let $\widetilde{g}$ be an element acting on the free product $\pi_1 (M) = \pi_1 (A) \ast \langle \tau \rangle$ such that the restriction of $\widetilde{g}$ on $\pi_1 (A)$ equals $g_A$, and $\widetilde{g}$ acts identically on $\langle \tau \rangle$. The element $\widetilde{g}$ leaves $X_-$ and $X_+$ invariants, and hence we can define an element $g$ on the quotient $\pi_1 (M) = \pi_1 (A) \ast \langle \tau \rangle / N$. This construction embeds $\Gamma_A$ as a subgroup of $\Gamma_M$.

We recall that in \cite{go1} Goldman showed that for the natural symplectic measure on $\X (A)$, a measurable function $f : \X (A) \longrightarrow \mathbb{R}$ that is $\Gamma_A$-invariant is almost everywhere equal to a function depending only on the traces of the boundary components $\{ X_+ , X_- , C_1 , \dots , C_m \} $. In our case, an element $[\rho_A]$ in $\phi (\X(M))$ satisfies $\tr (\rho_A( X_+)) = \tr (\rho_A (X_-))$. For a representation $[\rho]$ in $\X (M)$ we denote by $x, c_1 , \dots, c_m $ the traces of the elements $\rho (X)  , \rho(C_1) , \dots , \rho (C_m)$ respectively. So we infer from previous results, the following:

\begin{proposition}\label{inter}
Let $f : \X(M) \rightarrow \R$ be a $\Gamma_M$-invariant function. There exists a function $G \, : \, [-2,2]^{m+1} \rightarrow \mathbb{R}$ such that $f ([\rho]) = G(x,c_1,...,c_m)$ almost everywhere. 
\end{proposition}

To conclude the proof of Theorem \ref{thm:open} in the case of a non-orientable surface of even genus, we have to "eliminate" the coordinate $x$. At this point the surface needs to have a sufficiently large mapping class group in order to be able to find a Dehn twist that acts non-trivially on the coordinate $x$ corresponding to the trace of $X$. Thanks to the hypothesis on $\chi(M)$, there is a two-holed Klein bottle $N_{2,2}$ embedded in  $M$, such that $X$ is a non-separating two-sided curve in $N_{2,2}$. Now we study the particular case of $N_{2,2}$.

\subsection{The two-holed Klein Bottle}
\subsubsection{The character variety}

Let $M$ be a two-holed Klein bottle.


\begin{figure}[ht]\label{fig:N22}
\begin{center}
\scalebox{0.8} 
{
\begin{pspicture}(0,-2.936875)(13.844063,2.976875)
\psarc[linewidth=0.04](1.9040625,-0.576875){1.0}{90.0}{270.0}
\psline[linewidth=0.04cm](1.9040625,0.423125)(6.9040623,0.423125)
\psline[linewidth=0.04cm](1.9040625,-1.576875)(6.9040623,-1.576875)
\pscircle[linewidth=0.04,dimen=outer,fillstyle=crosshatch*,hatchwidth=0.04,hatchangle=0,hatchsep=0.14](2.1040626,-0.576875){0.45}
\pscircle[linewidth=0.04,dimen=outer,fillstyle=crosshatch*,hatchwidth=0.04,hatchangle=0,hatchsep=0.14](3.9040625,-0.576875){0.45}
\psellipse[linewidth=0.04,dimen=outer](6.9040623,0.023125)(0.2,0.4)
\psellipse[linewidth=0.04,dimen=outer](6.9040623,-1.176875)(0.2,0.4)
\psbezier[linewidth=0.04](6.9040623,-0.776875)(6.5040627,-0.776875)(6.1440625,-0.596875)(6.0240626,-0.376875)
\pscustom[linewidth=0.04]
{
\newpath
\moveto(6.8840623,-0.356875)
\lineto(6.7640624,-0.386875)
\curveto(6.7040625,-0.401875)(6.5790625,-0.421875)(6.5140624,-0.426875)
\curveto(6.4490623,-0.431875)(6.3590627,-0.461875)(6.3340626,-0.486875)
\curveto(6.3090625,-0.511875)(6.2740626,-0.541875)(6.2440624,-0.556875)
}
\pscircle[linewidth=0.04,dimen=outer](11.404062,-0.536875){2.4}
\pscircle[linewidth=0.04,dimen=outer,fillstyle=crosshatch*,hatchwidth=0.04,hatchangle=0,hatchsep=0.14](10.264063,0.383125){0.45}
\pscircle[linewidth=0.04,dimen=outer,fillstyle=crosshatch*,hatchwidth=0.04,hatchangle=0,hatchsep=0.14](10.6640625,-1.796875){0.45}
\pscircle[linewidth=0.04,dimen=outer](12.564062,0.503125){0.45}
\psbezier[linewidth=0.02](12.224063,-1.176875)(12.104062,-0.456875)(11.544063,-0.076875)(11.844063,0.743125)(12.144062,1.563125)(13.344063,1.223125)(13.184063,0.303125)(13.024062,-0.616875)(12.784062,-0.736875)(12.304063,-1.176875)
\psbezier[linewidth=0.02](12.224063,-1.136875)(11.384063,-0.596875)(11.204062,1.103125)(10.564062,0.763125)
\psbezier[linewidth=0.02](9.844063,0.203125)(8.844063,-0.216875)(10.724063,-0.276875)(12.264063,-1.176875)
\psbezier[linewidth=0.02](12.264063,-1.216875)(10.904062,-0.976875)(9.304063,-0.676875)(10.324062,-1.496875)
\psbezier[linewidth=0.02](10.984062,-2.136875)(11.624063,-2.756875)(12.004063,-2.356875)(12.264063,-1.176875)
\psdots[dotsize=0.16](12.264063,-1.196875)
\psline[linewidth=0.04cm](11.964063,-1.676875)(12.004063,-1.996875)
\psline[linewidth=0.04cm](12.004063,-2.016875)(12.284062,-1.756875)
\psline[linewidth=0.04cm](10.204062,-0.536875)(10.084063,-0.236875)
\psline[linewidth=0.04cm](10.104062,-0.256875)(10.424063,-0.236875)
\psline[linewidth=0.04cm](11.564062,0.603125)(11.884063,0.783125)
\psline[linewidth=0.04cm](11.884063,0.763125)(11.944062,0.463125)
\psline[linewidth=0.04cm](13.444062,-1.296875)(13.484062,-1.656875)
\psline[linewidth=0.04cm](13.484062,-1.656875)(13.824062,-1.476875)
\usefont{T1}{ptm}{m}{n}
\rput(12.430625,-2.046875){A}
\usefont{T1}{ptm}{m}{n}
\rput(9.83125,-0.546875){B}
\usefont{T1}{ptm}{m}{n}
\rput(11.73125,1.133125){C}
\usefont{T1}{ptm}{m}{n}
\rput(13.992657,-1.386875){K}
\end{pspicture} 
}

\end{center}
\caption{Two-holed Klein bottle}
\end{figure}
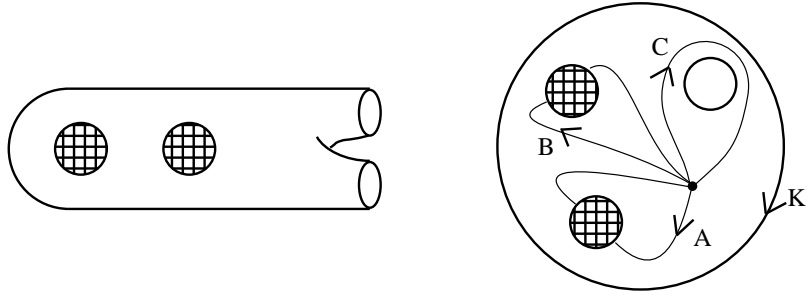

Its fundamental group is as follows
$$\pi = \pi_1(M)=\langle A , B , C , K \mid  A^2 B^2 C K^{-1} \rangle, $$
where $A, B, C, K$ are the curves drawn in Figure \ref{fig:N22}. So $\pi$ is a free group in three generators $A , B , C$. According to Magnus \cite{magnus}, we have trace coordinates on the space  $\X(M) = \Hom (\pi , \SU (2)) / \SU (2)$ given by the seven functions $a,b,c,d,x,y,z$ defined on $\X (M)$ by
$$\begin{array}{cccc} a = \tr (\rho (A)) ;& \,  b = \tr (\rho(B)) ;&\, c = \tr (\rho(C)) ;& \, x = \tr (\rho(AB)) ; \end{array}$$
$$\begin{array}{ccc}y = \tr (\rho(BC)) ;& \, z = \tr (\rho(CA))  ;& \, d = \tr (\rho(ABC)). \end{array} \mbox{       for any } [ \rho ] \in \X(M)$$ 
These seven coordinates satisfy the Fricke relation
\begin{equation}\label{fricke} a^2 +b^2 +c^2 +d^2 +x^2 +y^2 +z^2  - 
\end{equation}
$$((ab+cd)x + (bc+da)y+(ca+bd)z)+xyz+abcd-4 = 0$$


For any element $W\in \pi_1 (M)$ written as a word in $A,B,C,A^{-1},B^{-1},C^{-1}$, it is possible to compute the value of $\tr (\rho (W))$ in function of the seven coordinates using the simple formulas $\tr (P Q^{-1}) = \tr (P) \cdot \tr (Q) - \tr (PQ)$ and $\tr (P Q P^{-1}) = \tr (Q)$. These computations can be done algorithmically on a computer using a recursive program.

The elements $C ,K \in \pi$ correspond to the two boundary components of the surface. The character variety of the boundary of $M$ is
$$\X(\partial M) = \{ (c,k) \in [-2 , 2 ]^2 \}$$
where $k=\Tr (K) = \tr (A^2 B^2 C) = adb-az-by+c$. 

Let $X$ be the non-separating two-sided curve represented by $AB$ in $\pi$. The surface $M|X$ is a four-holed sphere. According to Proposition $\ref{inter}$, a measurable function $f : \X (N_{2,2}) \longrightarrow \mathbb{R}$ that is $\Gamma_M$-invariant, is almost everywhere equal to a function depending only on the coordinates $(x,c,k)$.

\subsubsection{The action of the twist about the curve $BBC$}

The curve $U$ shown in Figure \ref{fig:U} is represented by the element $U=BBC$ of $\pi_1(M)$. It is a simple two-sided curve, so the Dehn twist about $U$ can be defined.
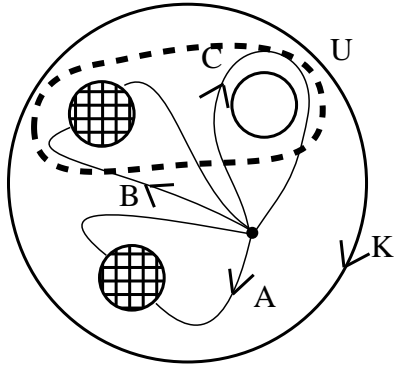
\begin{figure}[ht]\label{fig:U}
\begin{center}
\scalebox{1} 
{
\begin{pspicture}(0,-2.536875)(5.188125,2.576875)
\pscircle[linewidth=0.04,dimen=outer](2.748125,-0.136875){2.4}
\pscircle[linewidth=0.04,dimen=outer,fillstyle=crosshatch*,hatchwidth=0.04,hatchangle=0,hatchsep=0.14](1.608125,0.783125){0.45}
\pscircle[linewidth=0.04,dimen=outer,fillstyle=crosshatch*,hatchwidth=0.04,hatchangle=0,hatchsep=0.14](2.008125,-1.396875){0.45}
\pscircle[linewidth=0.04,dimen=outer](3.768125,0.903125){0.45}
\psbezier[linewidth=0.02](3.568125,-0.776875)(3.448125,-0.056875)(2.888125,0.323125)(3.188125,1.143125)(3.488125,1.963125)(4.508125,1.643125)(4.348125,0.763125)(4.188125,-0.116875)(4.008125,-0.236875)(3.648125,-0.776875)
\psbezier[linewidth=0.02](3.568125,-0.736875)(2.728125,-0.196875)(2.548125,1.503125)(1.908125,1.163125)
\psbezier[linewidth=0.02](1.188125,0.603125)(0.188125,0.183125)(2.068125,0.123125)(3.608125,-0.776875)
\psbezier[linewidth=0.02](3.608125,-0.816875)(2.248125,-0.576875)(0.648125,-0.276875)(1.668125,-1.096875)
\psbezier[linewidth=0.02](2.328125,-1.736875)(2.968125,-2.356875)(3.348125,-1.956875)(3.608125,-0.776875)
\psdots[dotsize=0.16](3.608125,-0.796875)
\psline[linewidth=0.04cm](3.308125,-1.276875)(3.348125,-1.596875)
\psline[linewidth=0.04cm](3.348125,-1.616875)(3.628125,-1.356875)
\psline[linewidth=0.04cm](2.308125,-0.456875)(2.188125,-0.156875)
\psline[linewidth=0.04cm](2.248125,-0.176875)(2.568125,-0.156875)
\psline[linewidth=0.04cm](2.908125,1.003125)(3.228125,1.183125)
\psline[linewidth=0.04cm](3.228125,1.163125)(3.288125,0.863125)
\psline[linewidth=0.04cm](4.788125,-0.896875)(4.828125,-1.256875)
\psline[linewidth=0.04cm](4.828125,-1.256875)(5.168125,-1.076875)
\usefont{T1}{ptm}{m}{n}
\rput(3.7746875,-1.646875){A}
\usefont{T1}{ptm}{m}{n}
\rput(1.9753125,-0.306875){B}
\usefont{T1}{ptm}{m}{n}
\rput(3.0753126,1.533125){C}
\usefont{T1}{ptm}{m}{n}
\rput(5.3367186,-0.986875){K}
\psbezier[linewidth=0.08,linestyle=dashed,dash=0.16cm 0.16cm](0.708125,0.603125)(0.708125,-0.196875)(1.6516049,0.019771751)(2.648125,0.103125)(3.644645,0.18647824)(4.69813,0.009193997)(4.528125,1.043125)(4.3581195,2.077056)(2.248125,1.583125)(1.868125,1.503125)(1.488125,1.423125)(0.708125,1.403125)(0.708125,0.603125)
\usefont{T1}{ptm}{m}{n}
\rput(4.7953124,1.633125){U}
\end{pspicture} 
}
\end{center}
\caption{The curve $U$}
\end{figure}

The Dehn twist $\tau_U$ about $U$ is given by the automorphism of $\pi_1(M)$
\begin{align*}
A \mapsto & \,\, A \\
B \mapsto & \,\, BBCBC^{-1}B^{-1}B^{-1} \\
C \mapsto & \,\, BBCB^{-1}B^{-1}
\end{align*}
The elements corresponding to $X,Y,Z,D$ and $K$ are transformed by $\tau_U$ as follows
\begin{align*}
X = AB \mapsto & \,\, ABBCBC^{-1}B^{-1}B^{-1} \\
Y = BC \mapsto & \,\, BBCB^{-1}\\
Z = CA \mapsto & \,\, BBCB^{-1}B^{-1}A \\
D = ABC \mapsto & \,\, ABBCB^{-1} \\
K = AABBC \mapsto & \,\, AABBC 
\end{align*}
This automorphism of $\pi_1 (M)$ is a lift of $\tau_U \in \mbox{Out} (\pi_1 (M))$ in $\mbox{Aut} (\pi_1 (M))$. Two lifts differ by an inner automorphism, which leaves invariant the conjugacy class and hence the trace coordinates of an element $[ \rho ] \in \X (M)$.
The transformation leaves invariant the conjugacy class of the elements corresponding to the boundary components, namely $C$ and $K$.

The twist $\tau_U$ induces an action on $\X_{\mathcal{C}} (M)$ which can be seen on the coordinates $(a,b,x,y,z,d) \in [-2,2]^6$ as :
\begin{align*}
a \mapsto & a \\
b \mapsto & b \\
x \mapsto & b^2 xy^2 +b^2 yz-b^3 dy - aby^2 +b^2 cd+c^2 x \\
 & -2bcxy+2bdy-bcz+acy-yz-cd+ab-x \\
y \mapsto & y \\
d \mapsto & b^2 d-bxy-bz+cx+ay-d \\
z \mapsto & b (b^2 d-bxy-bz+cx+ay-d) - bd + z 
\end{align*}

 For non-zero $a,b$ , we can write
\begin{equation}
d = \frac{az+by -c +k}{ab} .
\end{equation}
We replace in $(\ref{fricke})$ the coordinate $d$ with its expression in function of $k$ which is $\tau_U$-invariant. Then the equation  $(\ref{fricke})$ becomes:
\begin{equation}\label{fri2}
x^2 + (by-c)x \left( \frac{z}{b} \right) + \left( \frac{z}{b} \right)^2 + 2 D x +2Ez +F = 0
\end{equation}
with\\
$$D = \dfrac{a^2 b^2 -c^2 +ck+bcy}{-2ab} , E = \dfrac{2 c - 2 k - 2 b y - b^2  c + b^2  k + b^3 y + a^2   b y}{-2ab^2 },$$
$$F = \dfrac{1}{a^2 b^2 }(k^2 +(by-c)^2 -(a^2 -2)(by-c)k + a^2 b^2 (a^2 +b^2 +ck-4) - a^2 ck+a^2 bcy).$$

\begin{remark}
The set of representations $[ \rho ]$ such that $a$ and $b$ are zero, is a null measure subset of $\X (N_{2,2})$. So it suffices to prove ergodicity in the complementary of this subset to have ergodicity on the whole space.
\end{remark}

We make a change of variable $z^{'} = \frac{z}{b}$, and the equation $(\ref{fri2})$ becomes:
\begin{equation}\label{ell}
x^2 + (by-c)x z' + z'^2 + 2 D' x +2E'z +F = 0,
\end{equation}
with
$$D' = D , E' = \frac{(b^2 -2)(by+k-c)+a^2 by}{-2ab} .$$
We denote by $u \in [-2 , 2]$ the trace $\Tr (\rho (U)) = by-c $. When $u \notin \{ -2 , 2\}$, we rewrite (\ref{fri2}) as
$$\frac{2+u}{4} ( (x+z' ) - (x_0 (u) +z'_0 (u) )^2 + \frac{2-u}{4} ( (x-z' ) - (x_0 (u) - z'_0 (u))^2 = R $$
where $x_0$, $z_0$ and $R$ are functions in $a,b,c,y,k$, and therefore are $\tau_U$-invariants. The expression of $R$ is the following:
$$ R = \frac{(b^2 +c^2 +y^2 -bcy-4)((a^2 -2)^2 +u^2 +k^2 -(a^2 -2)uk -4)}{a^2 (4-u^2 )}.$$
For a particular value of $-2 < u < 2 $, the left term of the equation is a quadratic function on $x,z'$ with positive coefficients. The positivity of the right term is given by the following fact concerning representations of the free group in two generators, and we refer to \cite{goldman88,nielsen,magnus} for proofs.
\begin{lemma}\label{frirel}
Let $\mathbb{F}_2$ be the free group in two generators $P$ and $Q$. Let $\mathcal{X} (M)$ be the space of conjugacy classes of representations $\Hom(\mathbb{F}_2 , \SU (2)) / \SU (2)$. Then the trace map 
$$
\begin{array}{ccc}
\mathcal{X} (M) & \longrightarrow & \R^3 \\
\left[ \rho \right] &\longmapsto & \left( \begin{array}{c} \tr (\rho (P)) \\ \tr (\rho (Q)) \\ \tr (\rho (PQ)) \end{array} \right)
\end{array}$$
identifies $\mathcal{X} (M)$ with the set
$$\mathfrak{B} =  \left\{ ( p , q , r) \in [-2 , 2] ^3 \, \big| \, p^2 + q^2 + r^2 - pqr - 4 \leq 0 \right\}$$
\end{lemma}

In the right term $R$ we recognize $(b,c,y)$ and $(a^2 -2,u,k)$  as the characters of representations in $\SU (2)$ of the free groups in two generators $\langle B , C \rangle$ and $\langle AA , BBC \rangle$ respectively, as we have $ BC = Y$ and $AABBC = K $. So according to Lemma \ref{frirel}, we have the following inequalities:
\begin{align}
&(b^2 +c^2 +y^2 -bcy-4) \leq 0  ,\\
&((a^2 -2)^2 +u^2 +k^2 -(a^2 -2)uk-4)  \leq 0  .
\end{align}

Moreover $4 - u^2 > 0$, so $R$ is non-negative. Hence the set of all $(x,z')$ satisfying the equation (\ref{ell}) corresponds to an ellipse. This exhibits $\X_{\mathcal{C}} (M)$ as a family of ellipses $E_{\mathcal{C}} (M) (a,b,y)$ that are parametrized by $(c,k,a,b,y)$. Now we express the action of $\tau_U$ on $(x,z')$ using $(x_0 , z'_0)$, as follows:
$$ \left[ \begin{array}{c} x \\ z' \end{array} \right] \mapsto
 \left[ \begin{array}{c} x_0(u) \\ z'_0(u) \end{array} \right] +
\left[ \begin{array}{cc} u^2-1 & u \\ -u & -1 \end{array} \right] \cdot
\left( \left[ \begin{array}{c} x \\ z' \end{array} \right] -
 \left[ \begin{array}{c} x_0(u) \\ z'_0(u) \end{array} \right] \right)$$
 
 This transformation is a rotation of angle $2 \theta_U = 2 \cos ^{-1} (\tr(\rho (U))/2)$ on the ellipse $E_{\mathcal{C}} (M) (a,b,y)$ for fixed $(c,k,a,b,y)$. For all boundary traces $(c,k)$ and for almost all $(a,b,y)$, the angle $\theta_U$ is an irrational multiple of $\pi$. So for almost all $(a,b,y)$, the action of $\tau_U$ is ergodic on the ellipse $E_{\mathcal{C}} (M) (a,b,y)$.
 
Let $f : \X(M) \rightarrow \R$ be a $\Gamma_M$-invariant measurable function. The function $f$ is $\tau_U$-invariant, and by the Lemma of ergodic decomposition, there exists a function $H \, : \, [-2,2]^5 \longrightarrow \mathbb{R}$ such that $f([\rho]) = H(c,k,a,b,y)$ almost everywhere.
On the other hand, according to Proposition $\ref{inter}$, there exists $G \, : \, [-2 , 2 ]^3 \rightarrow \R$ such that $f([\rho]) = G(c,k,x)$ almost everywhere. Therefore the function $f$ depends only on the traces $(c,k)$ of the boundary components . This ends the proof of the Theorem $\ref{thm:open}$ in the case of a two-holed Klein bottle.

\subsection{Conclusion}

We can now prove the case of a non-orientable surface of even genus. Let $M$ be the surface $N_{g,m}$, with $\chi (M) \leq -2$. We can find an embedding of a two-holed Klein bottle $S = N_{2,2}$ in $M$. Let $X$ be the non-separating two-sided circle on $S$ such that the surface $M|X$ is an orientable surface. The Proposition \ref{inter} states that for any $\Gamma_M$-invariant function $f : \X(M) \rightarrow \mathbb{R}$ there is a function $G \, : \, [-2,2]^{m+1} \rightarrow \mathbb{R}$ such that $f ([\rho]) = G(x,c_1,...,c_m)$ almost everywhere.

The mapping class group $\Gamma_S$ of the two-holed Klein bottle can be seen as a subgroup of $\Gamma_M$. The restriction map $\X (M) \longrightarrow \X (S)$ is $\Gamma_S$-equivariant.  The ergodicity of $\Gamma_S$ on the relative character variety of $S$ proves that the function $f$ is almost everywhere equal to a function that does not depend on the function $x = \tr (\rho X)$. The two arguments combine to prove that a $\Gamma_M$-invariant function is almost everywhere equal to a function depending only on the traces of the boundaries $\mathcal{C} = (c_1 , ... , c_m)$. Hence, the action of $\Gamma_M$ is ergodic on $\X_{\mathcal{C}} (M)$ and this ends the proof of the theorem \ref{thm:open} in the case of a non-orientable surface of even genus.

\section{Non-orientable surfaces of odd genus with Euler characteristic $-2$ }\label{section:-2}

In this section, we study the case where $M$ is a three-holed projective plane $N_{1,3}$ or the one-holed non-orientable surface of genus three $N_{3,1}$. In both cases, the fundamental group is isomorphic to the free group in three generators. Hence, we will use the trace coordinates defined in the previous section.

\subsection{The character variety of $N_{1,3}$}

Let $M$ be the surface $N_{1,3}$, and let $B,C$ and $K$ be its three boundary components.

\begin{figure}[ht]\label{fig:N13}
\begin{center}
\scalebox{1} 
{
\begin{pspicture}(0,-2.516875)(5.3840623,2.556875)
\definecolor{color812b}{rgb}{0.8,0.8,0.8}
\pscircle[linewidth=0.04,dimen=outer](2.9840624,-0.116875){2.4}
\pscircle[linewidth=0.04,dimen=outer,fillstyle=crosshatch*,hatchwidth=0.04,hatchangle=0,hatchsep=0.14,fillcolor=color812b](1.4440625,-0.056875){0.45}
\pscircle[linewidth=0.04,dimen=outer](3.7440624,-1.196875){0.45}
\pscircle[linewidth=0.04,dimen=outer](3.7640624,0.883125){0.45}
\psbezier[linewidth=0.02](2.6840625,-0.056875)(2.8040626,0.483125)(2.8640625,0.983125)(3.1640625,1.363125)(3.4640625,1.743125)(4.5640626,1.763125)(4.4240627,0.843125)(4.2840624,-0.076875)(3.4040625,0.203125)(2.7640624,-0.056875)
\psbezier[linewidth=0.02](2.7240624,-0.056875)(2.2840624,0.803125)(1.3840625,1.383125)(1.4240625,0.403125)
\psbezier[linewidth=0.02](1.4240625,-0.456875)(1.4640625,-1.516875)(2.3240626,-0.676875)(2.7240624,-0.056875)
\psdots[dotsize=0.16](2.7440624,-0.056875)
\psline[linewidth=0.04cm](2.0440626,-0.576875)(1.9640625,-0.876875)
\psline[linewidth=0.04cm](1.9640625,-0.896875)(2.2240624,-0.856875)
\psline[linewidth=0.04cm](2.7840624,1.063125)(3.0240624,1.183125)
\psline[linewidth=0.04cm](3.0440626,1.183125)(3.0440626,0.923125)
\psline[linewidth=0.04cm](3.9840624,-0.276875)(4.0640626,-0.576875)
\psline[linewidth=0.04cm](4.0240626,-0.536875)(3.7440624,-0.516875)
\psline[linewidth=0.04cm](4.9040623,-1.096875)(4.9440627,-1.456875)
\psline[linewidth=0.04cm](4.9440627,-1.456875)(5.2840624,-1.276875)
\usefont{T1}{ptm}{m}{n}
\rput(2.130625,-1.206875){A}
\usefont{T1}{ptm}{m}{n}
\rput(2.71125,1.273125){B}
\usefont{T1}{ptm}{m}{n}
\rput(4.43125,-0.146875){C}
\usefont{T1}{ptm}{m}{n}
\rput(5.652656,-1.286875){K}
\psbezier[linewidth=0.02](2.7240624,-0.056875)(3.3240626,-0.276875)(4.0040627,-0.176875)(4.3440623,-0.996875)(4.6840625,-1.816875)(3.6840625,-2.276875)(3.1440625,-1.596875)(2.6040626,-0.916875)(2.8640625,-0.636875)(2.7240624,-0.096875)
\end{pspicture} 
}
\end{center}
\caption{The three-holed projective plane}
\end{figure}
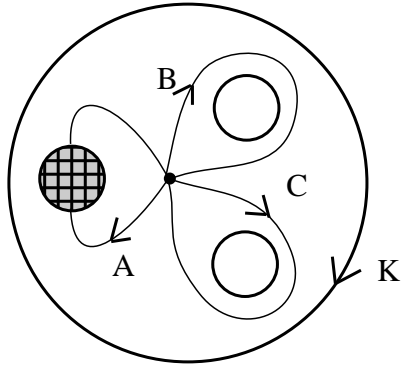
Its fundamental group admits the following presentation
$$\pi = \pi_1(M)=\langle A , B , C , K \mid  A^2 B C K^{-1} \rangle .$$
where $A,B,C,K$ are the curves drawn in Figure \ref{fig:N13}. We see that $\pi$ is a free group on three generators $A,B,C$, so the coordinates on the space $\X (M)$ are given by the seven functions $a,b,c,d,x,y,z$ defined as previously. These seven functions satisfy the Fricke relation (\ref{fricke}). The character of the boundary of $M$ is 
$$\X (\partial M) = (b,c,k)\, ,$$
where $k=\tr(\rho(K)) = \tr (\rho (A^2BC)) = ad-y$. We replace $y$ in the equation (\ref{fricke}) with its expression in function of $a,d$ and $k$. The equation then becomes:
$$a^2+b^2+c^2+d^2+x^2+z^2+k^2 - $$
\begin{equation}\label{star}
((ab+cd)x + (xz-bc-da)k+(ca+bd)z)+adxz-4 = 0 .
\end{equation}
For a fixed character of the boundary $\mathcal{C}=(b,c,k)$, the character variety relative to $\mathcal{C}$ is given by
$$\X_{\mathcal{C}} (M) = \{ (a,x,z,d) \in [-2 , 2]^4 \mid (a,x,z,d) \mbox{ satisfies } (\ref{star}) \} .$$
The complete character variety can be expressed as
$$\X (M) = \bigcup_{-2 \leq b,c,k \leq 2} \X_{\mathcal{C} } (M).$$
The Theorem \ref{thm:open} becomes in this particular case
\begin{proposition}\label{thpp3} For all boundary components $\mathcal{C}=(b,c,k)\in ] -2 , 2 [^3$, the action of $\Gamma_M$ on $\X_{\mathcal{C}} (M)$ is ergodic.
\end{proposition}

\subsection{The action of Dehn twists}
\subsubsection{The twist about the curve $T = AAB$}

The curve $T$ shown in Figure 6 is represented by the element $T=AAB$ in $\pi_1(M)$. It is a two-sided circle, so the Dehn twist about $T$ can be defined.
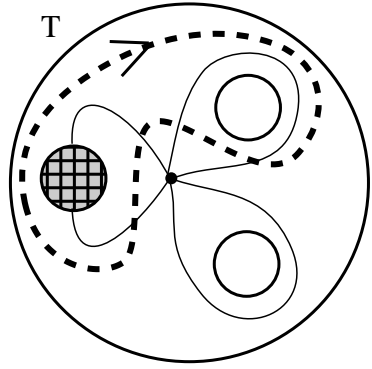
\begin{figure}[ht]
\begin{center}
\scalebox{1} 
{
\begin{pspicture}(0,-2.5142188)(5.3690624,2.5542188)
\definecolor{color812b}{rgb}{0.8,0.8,0.8}
\pscircle[linewidth=0.04,dimen=outer](2.9690626,-0.11421875){2.4}
\pscircle[linewidth=0.04,dimen=outer,fillstyle=crosshatch*,hatchwidth=0.04,hatchangle=0,hatchsep=0.14,fillcolor=color812b](1.4290625,-0.05421875){0.45}
\pscircle[linewidth=0.04,dimen=outer](3.7290626,-1.1942188){0.45}
\pscircle[linewidth=0.04,dimen=outer](3.7490625,0.8857812){0.45}
\psbezier[linewidth=0.02](2.6690626,-0.05421875)(2.7890625,0.48578125)(2.8490624,0.98578125)(3.1490624,1.3657813)(3.4490626,1.7457813)(4.5490627,1.7657813)(4.4090624,0.84578127)(4.2690625,-0.07421875)(3.3890624,0.20578125)(2.7490625,-0.05421875)
\psbezier[linewidth=0.02](2.7090626,-0.05421875)(2.2690625,0.80578125)(1.3690625,1.3857813)(1.4090625,0.40578124)
\psbezier[linewidth=0.02](1.4090625,-0.45421875)(1.4490625,-1.5142188)(2.3090625,-0.6742188)(2.7090626,-0.05421875)
\psdots[dotsize=0.16](2.7290626,-0.05421875)
\psbezier[linewidth=0.02](2.7090626,-0.05421875)(3.3090625,-0.27421874)(3.9890625,-0.17421874)(4.3290625,-0.99421877)(4.6690626,-1.8142188)(3.6690626,-2.2742188)(3.1290624,-1.5942187)(2.5890625,-0.9142187)(2.8490624,-0.63421875)(2.7090626,-0.09421875)
\psbezier[linewidth=0.08,linestyle=dashed,dash=0.16cm 0.16cm](0.7890625,-0.41421875)(0.5890625,0.4222036)(1.0719324,1.0057813)(1.8436042,1.4457812)(2.6152759,1.8857813)(3.4480577,1.9886552)(4.1180053,1.7257812)(4.7879524,1.4629073)(4.8490624,0.74578124)(4.402305,0.26578125)(3.9555478,-0.21421875)(3.0011117,0.74578124)(2.5340474,0.7057812)(2.0669827,0.66578126)(2.4027343,-0.9610909)(2.08729,-1.1742188)(1.7718458,-1.3873466)(0.9890625,-1.2506411)(0.7890625,-0.41421875)
\usefont{T1}{ptm}{m}{n}
\rput(1.1290625,1.9557812){T}
\psline[linewidth=0.04cm](1.9090625,1.7657813)(2.4290626,1.7457813)
\psline[linewidth=0.04cm](2.4490626,1.7257812)(2.0690625,1.3057812)
\end{pspicture} 
}
\end{center}
\caption{The curve $T$}
\end{figure}

The Dehn twist $\tau_T$ about $T$ is given by the following automorphism of $\pi_1(M)$
\begin{align*}
A \mapsto & \,\, AABAB^{-1}A^{-1}A^{-1} \\
B \mapsto & \,\, AABA^{-1}A^{-1} \\
C \mapsto & \,\, C
\end{align*}
The curves corresponding to $X,K,Z$ and $D$ are transformed by $\tau_T$ as follows
\begin{align*}
X  = AB \mapsto & \,\, AABA^{-1} \\
Z  = CA \mapsto &  \,\, C AABAB^{-1}A^{-1}A^{-1}\\
D  = ABC \mapsto & \,\, A ABA^{-1} C \\
K  = AABC \mapsto & \,\, AABC
\end{align*}
This transformation leaves invariant the boundary character $\mathcal{C}$. So $\tau_T$ induces an action on $\X_{\mathcal{C}} (M)$ which can be seen on the coordinates $(a,x,z,d) \in ]-2 , 2 [^4$ as:

\begin{align*}
a \mapsto & a \\
x \mapsto & x \\
z \mapsto & a^2 x^2 z-a^2 kx-a^2 cx+b^2 z-2abxz\\
& +axd+bcx+abk+kx-bc+ac-z \\
d \mapsto & ak-axz+cx+bz-d 
\end{align*}
The coordinates $a$ and $x$ are $\tau_T$-invariant. We denote $t = \Tr (\rho (T)) = ax-b $. When $t \neq \pm 2$, we rewrite (\ref{star}) as:
{\small
\begin{equation}\label{eqnT}
\frac{2+t}{4} \left( (d+z)- \frac{(a+x)(c+k)}{t+2} \right)^2 + \frac{2-t}{4} \left( (d-z) - \frac{(a-x)(k-c)}{t-2} \right)^2 = R_T 
\end{equation}
}
with
$$R_T := \frac{(t^2 +c^2 +k^2 -tck-4)(a^2 +b^2 +x^2 -abx-4)}{4-t^2 }.$$
The function $R_T$ is also $\tau_T$-invariant. For a fixed value of $t$ in $] -2 , 2[ $, the left term of the equation $(\ref{eqnT} )$ is a quadratic function of $d$ and $z$ with positive coefficients. In the right term $R_T$ we recognize $(a,b,x)$ and $(t,c,k)$ as the characters of representations in $\SU (2)$ of the free groups in two generators $\langle A , B \rangle$ and $\langle T , C \rangle$ respectively, as we have $AB = X$ and $TC = K$. So according to Lemma \ref{frirel}, we have the following inequalities:
\begin{align}
(a^2 +b^2 +x^2 -abx-4) & \leq 0, \\
(t^2 +c^2 +k^2 -tck-4) & \leq 0.
\end{align}

Moreover $4 - t^2  > 0$, so that $R_T \geq 0$. So the set of coordinates $(d, z)$ satisfying the equation (\ref{eqnT} ) corresponds to an ellipse. For fixed values of $b,c,k,a,x$, the intersection 
\begin{align*}
  E_{\mathcal{C}} & : = \X_{\mathcal{C}} (M) \cap (\{ a,x \} \times \R ^2 ) \\
   & = \{ a,x \} \times \left\{ (d,z) \in \mathbb{R} ^2 \mid (d,z) \mbox{ satisfies } \ref{eqnT}  \right\}
   \end{align*}
is an ellipse preserved by $\tau_T$. This exhibits $\X_{\mathcal{C}} (M)$ as a family of ellipses $E_{\mathcal{C}} (M) (a,x)$ parametrized by $(b,c,k,a,x)$. We can rewrite the equation in the following way :
$$Q_t ( d-d_0 (t) , z - z_0(t) ) = R_T$$
with
\begin{align}
d_0 (\nu) & = ac+xk- \nu (ak+cx) \\
z_0 (\nu) & = ak+xc \\
Q_{\nu}(\eta,\zeta) & = \frac{\eta^2+\zeta^2 - \nu \eta\zeta }{4-\nu^2}
\end{align}
Now we express the action of $\tau_T$ on $d,z$ in terms of $d_0 , z_0$, which gives us (after simplification) :
\begin{equation} \left[ \begin{array}{c} d \\ z \end{array} \right] \mapsto \left[ \begin{array}{c} d_0(t) \\ z_0(t) \end{array} \right] +\left[ \begin{array}{cc} -1 & -t \\ t & t^2 -1 \end{array} \right] \cdot\left( \left[ \begin{array}{c} d \\ z \end{array} \right] -  \left[ \begin{array}{c} d_0(t) \\ z_0(t) \end{array} \right] \right) . \end{equation}

This transformation is the rotation of angle $- \theta_T = -2 \cos ^{-1} (t/2)$ on the ellipse $E_{\mathcal{C}} (M) (a,x)$ defined for $(b,c,k,a,x)$ fixed. In particular, for fixed boundary traces $(b,c,k)$ and for almost all $(a,x)$, the angle $\theta_T$ is an irrational multiple of $\pi$ . So for almost all $(a,x)$, the action of $\tau_T$ is ergodic on the ellipse $E_{\mathcal{C}} (M) (a,x)$.

\subsubsection{The twist about the curve $U = CAA$}

The curve $U$ shown in Figure 7 is represented by the element $U=CAA$ in $\pi_1(M)$. It is a simple two-sided circle, so the Dehn twist about $U$ can be defined.
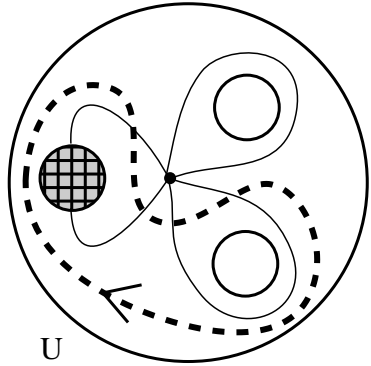
\begin{figure}[ht]
\begin{center}
\scalebox{1} 
{
\begin{pspicture}(0,-2.5142188)(5.388125,2.5542188)
\definecolor{color812b}{rgb}{0.8,0.8,0.8}
\pscircle[linewidth=0.04,dimen=outer](2.988125,-0.11421875){2.4}
\pscircle[linewidth=0.04,dimen=outer,fillstyle=crosshatch*,hatchwidth=0.04,hatchangle=0,hatchsep=0.14,fillcolor=color812b](1.448125,-0.05421875){0.45}
\pscircle[linewidth=0.04,dimen=outer](3.748125,-1.1942188){0.45}
\pscircle[linewidth=0.04,dimen=outer](3.768125,0.8857812){0.45}
\psbezier[linewidth=0.02](2.688125,-0.05421875)(2.808125,0.48578125)(2.868125,0.98578125)(3.168125,1.3657813)(3.468125,1.7457813)(4.568125,1.7657813)(4.428125,0.84578127)(4.288125,-0.07421875)(3.408125,0.20578125)(2.768125,-0.05421875)
\psbezier[linewidth=0.02](2.728125,-0.05421875)(2.288125,0.80578125)(1.388125,1.3857813)(1.428125,0.40578124)
\psbezier[linewidth=0.02](1.428125,-0.45421875)(1.468125,-1.5142188)(2.328125,-0.6742188)(2.728125,-0.05421875)
\psdots[dotsize=0.16](2.748125,-0.05421875)
\psbezier[linewidth=0.02](2.728125,-0.05421875)(3.328125,-0.27421874)(4.008125,-0.17421874)(4.348125,-0.99421877)(4.688125,-1.8142188)(3.688125,-2.2742188)(3.148125,-1.5942187)(2.608125,-0.9142187)(2.868125,-0.63421875)(2.728125,-0.09421875)
\usefont{T1}{ptm}{m}{n}
\rput(1.1753125,-2.3242188){U}
\psline[linewidth=0.04cm](1.848125,-1.5742188)(2.188125,-1.9742187)
\psline[linewidth=0.04cm](2.368125,-1.4742187)(1.888125,-1.5742188)
\psbezier[linewidth=0.08,linestyle=dashed,dash=0.16cm 0.16cm](0.828125,-0.05421875)(0.868125,0.72578126)(1.328125,1.3457812)(1.948125,1.1457813)(2.568125,0.94578123)(1.9840757,-0.30682507)(2.508125,-0.59421873)(3.0321743,-0.8816124)(3.638912,-0.076725245)(4.068125,-0.13421875)(4.497338,-0.19171226)(4.908125,-0.99421877)(4.588125,-1.6342187)(4.268125,-2.2742188)(3.274884,-2.1898754)(2.348125,-1.8142188)(1.4213661,-1.4385622)(0.788125,-0.83421874)(0.828125,-0.05421875)
\end{pspicture} 
}
\end{center}
\caption{The curve $U$}
\end{figure}

The Dehn twist $\tau_U$ about $U$ is given by the automorphism of $\pi_1(M)$
\begin{align*}
A \mapsto & \,\, CAAAA^{-1}A^{-1}C^{-1} = C A C^{-1} \\
B \mapsto & \,\, B \\
C \mapsto & \,\, CAACA^{-1}A^{-1}C^{-1}\\
\end{align*}
The curves corresponding to $X,K,Z$ and $D$ map to :
\begin{align*}
X = AB \mapsto & \,\, CAC^{-1}B \\
Z = CA \mapsto & \,\, C AACA^{-1}C^{-1}\\
D = ABC \mapsto & \,\, CAC^{-1}BCAACA^{-1}A^{-1}C^{-1} \\
K = AABC \mapsto & \,\, CAAC^{-1}BCAACA^{-1}A^{-1}C^{-1}
\end{align*}

We easily check that this transformation leaves invariant the boundary character $\mathcal{C}$. So $\tau_U$ induces an action on $\X_{\mathcal{C}} (M)$ which can be seen on the coordinates $(a,x,z,d) \in ]-2 , 2 [^4$
\begin{align*}
a \mapsto & a \\
x \mapsto & ab-zy+cd-x \\
z \mapsto & z \\
d \mapsto & a d - a y + a b c - a c x + a b z - a x z + a c d z - a c y z +\\
  & a^2  d + a c^2   d - a y z^2  
\end{align*}
The coordinates $a$ and $z$ are $\tau_U$-invariant. We denote $u = \Tr (\rho (U)) = az-c $ and when $u \neq \pm 2$, and we rewrite (\ref{star}) as
$$Q_u ( d-d_0 (u) , x - x_0(u) ) = R_U$$
with $u:=\Tr U = az-c$ which is $\tau_U$-invariant and 
\begin{align}
d_0 (\nu) & = ab+zk- \nu (ak+bz) \\
x_0 (\nu) & = ak+zb \\
R_U & = \frac{(u^2 +b^2 +k^2 -ubk-4)(a^2 +c^2 +z^2 -acz-4)}{4-u^2 }
\end{align}

The function $R_U$ is also $\tau_U$-invariant. For fixed value of $-2 < u < 2 $, the left term of the equation is a quadratic function of $d$ and $x$ with positive coefficients. In the right term $R_U$ we recognize $(a,c,z)$ and $(u,b,k)$ as the characters of representations in $\SU (2)$ of the free groups in two generators $\langle C , A \rangle$ and $\langle U , B \rangle$ respectively, as we have $ CA = Z$ and $UB = C K C^{-1}$. So according to Lemma \ref{frirel}, we have:
\begin{align}
(a^2 +c^2 +z^2 -acz-4) & \leq 0, \\
(u^2 +b^2 +k^2 -ubk-4) & \leq 0.
\end{align}
Moreover $4 - u^2  > 0$, so that $R_U \geq 0$. So the set of coordinates $d$ and $x$ satisfying the equation corresponds to an ellipse. This exhibits $\X_{\mathcal{C}} (M)$ as a family of ellipses $E_{\mathcal{C}} (M) (a,z)$ parametrized by $(b,c,k,a,z)$. Now we express the transformation $\tau_U$ on the coordinates $d,x$ in terms of $d_0 , x_0$, which gives us :
\begin{equation} \left[ \begin{array}{c} d \\ x \end{array} \right] \mapsto \left[ \begin{array}{c} d_0(u) \\ x_0(u) \end{array} \right] + \left[ \begin{array}{cc} -1 & -u \\ u & u^2 -1 \end{array} \right] \cdot\left( \left[ \begin{array}{c} d \\ x \end{array} \right] - \left[ \begin{array}{c} d_0(u) \\ x_0(u) \end{array} \right] \right)  . \end{equation}

This transformation is the rotation of angle $- \theta_U = -2 \cos ^{-1} (u/2)$ on the ellipse $E_{\mathcal{C}} (M) (a,z)$ defined earlier for fixed $(b,c,k,a,z)$. In particular, for fixed boundary traces $(b,c,k)$ and for almost all $(a,z)$, $\theta_U$ is an irrational multiple of $\pi$. So for almost all $(a,z)$, the action of $\tau_U$ is ergodic on the ellipse $E_{\mathcal{C}} (M) (a,z)$.

\subsubsection{The twist about $W$}

The curve $W$ shown in Figure \ref{fig:W} is represented by the element $W=CAB^{-1}A^{-1}$ in $\pi_1(M)$. It is a simple two-sided circle, so the Dehn twist about $W$ can be defined.
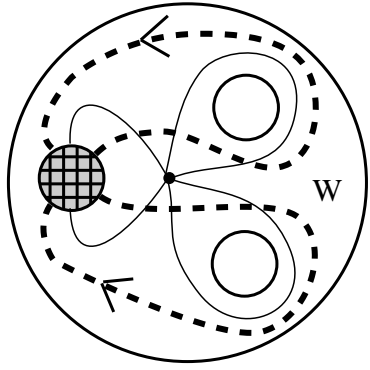
\begin{figure}[ht]\label{fig:W}
\begin{center}
\scalebox{1} 
{
\begin{pspicture}(0,-2.7442188)(5.8453126,2.7842188)
\definecolor{color812b}{rgb}{0.8,0.8,0.8}
\pscircle[linewidth=0.04,dimen=outer](3.4453125,-0.34421876){2.4}
\pscircle[linewidth=0.04,dimen=outer,fillstyle=crosshatch*,hatchwidth=0.04,hatchangle=0,hatchsep=0.14,fillcolor=color812b](1.9053125,-0.28421876){0.45}
\pscircle[linewidth=0.04,dimen=outer](4.2053127,-1.4242188){0.45}
\pscircle[linewidth=0.04,dimen=outer](4.2253127,0.65578127){0.45}
\psbezier[linewidth=0.02](3.1453125,-0.28421876)(3.2653124,0.25578126)(3.3253126,0.75578123)(3.6253126,1.1357813)(3.9253125,1.5157813)(5.0253124,1.5357813)(4.8853126,0.61578125)(4.7453127,-0.30421874)(3.8653126,-0.02421875)(3.2253125,-0.28421876)
\psbezier[linewidth=0.02](3.1853125,-0.28421876)(2.7453125,0.5757812)(1.8453125,1.1557813)(1.8853126,0.17578125)
\psbezier[linewidth=0.02](1.8853126,-0.68421876)(1.9253125,-1.7442187)(2.7853124,-0.90421873)(3.1853125,-0.28421876)
\psdots[dotsize=0.16](3.2053125,-0.28421876)
\psbezier[linewidth=0.02](3.1853125,-0.28421876)(3.7853124,-0.50421876)(4.4653125,-0.40421876)(4.8053126,-1.2242187)(5.1453123,-2.0442188)(4.1453123,-2.5042188)(3.6053126,-1.8242188)(3.0653124,-1.1442188)(3.3253126,-0.8642188)(3.1853125,-0.32421875)
\usefont{T1}{ptm}{m}{n}
\rput(5.311406,-0.45421875){W}
\psbezier[linewidth=0.08,linestyle=dashed,dash=0.16cm 0.16cm](2.2853124,-0.52421874)(2.8053124,-0.9242188)(4.5232244,-0.44127154)(4.9453125,-0.82421875)(5.3674006,-1.207166)(5.0538015,-2.1503367)(4.4453125,-2.3242188)(3.8368237,-2.4981008)(1.6968781,-1.4086272)(1.6653125,-1.3735399)(1.633747,-1.3384528)(1.4053125,-0.956338)(1.6253124,-0.6384698)
\psbezier[linewidth=0.08,linestyle=dashed,dash=0.16cm 0.16cm](2.2253125,0.03578125)(2.4853125,0.29578125)(2.745383,0.30765373)(3.1053126,0.33578125)(3.465242,0.36390877)(4.2148633,-0.19064023)(4.6253123,-0.14421874)(5.0357614,-0.09779728)(5.225741,0.40850857)(5.1253123,0.91578126)(5.0248837,1.423054)(4.5450716,1.657719)(3.5453124,1.6357813)(2.5455532,1.6138434)(1.9018767,1.0537882)(1.9053125,1.0557812)(1.9087483,1.0577743)(1.3053125,0.41578126)(1.6453125,0.11578125)
\psline[linewidth=0.04cm](2.4653125,-2.0642188)(2.3053124,-1.6842188)
\psline[linewidth=0.04cm](2.3253126,-1.6642188)(2.7253125,-1.6242187)
\psline[linewidth=0.04cm](3.1853125,1.8757813)(2.8253126,1.5557812)
\psline[linewidth=0.04cm](2.8253126,1.5557812)(3.2053125,1.3357812)
\end{pspicture} 
}
\end{center}
\caption{The curve $W$}
\end{figure}


The Dehn twist $\tau_U$ about $U$ is given by the following automorphism of $\pi_1(M)$
\begin{align*}
A \mapsto & \,\, CAB^{-1}A^{-1}C^{-1}AB \\
B \mapsto & \,\, (B^{-1}A^{-1}CA)B(A^{-1}C^{-1}AB )\\
C \mapsto & \,\, CAB^{-1}A^{-1}CABA^{-1}C^{-1} = W C W^{-1}
\end{align*}
The curves corresponding to $X,K,Z$ and $D$ are transformed by $\tau_W$ as follows:
\begin{align*}
X  = AB \mapsto & \,\, AB \\
Z  = CA \mapsto & \,\, CA\\
D  = ABC \mapsto & \,\, ABCAB^{-1}A^{-1}CABA^{-1}C^{-1} \\
K  = AABC \mapsto & \,\, (CAB^{-1}A^{-1}C^{-1}AB ) ABCA (B^{-1}A^{-1}CABA^{-1}C^{-1}) 
\end{align*}
We easily check that this transformation leaves invariant the boundary character $\mathcal{C}$. So $\tau_W$ induces an action on $\X_{\mathcal{C}} (M)$ which can be seen on the coordinates $(a,x,z,d) \in ]-2 , 2 [^4$ as follows:
\begin{align*}
a \mapsto & w(xc-d)-(x(cw-b)-(zc-a)) \\
x \mapsto & x \\
z \mapsto & z \\
d \mapsto & w(dw-(zc-a))-(x(wb-c)-(zb-d)) ,
\end{align*}
where $w = \Tr W = xz-k$. The coordinates $x$ and $z$ are $\tau_W$-invariant and when $w \neq \pm 2$ we rewrite (\ref{star}) as
$$Q_w ( a-a_0 (w) , d - d_0(w) ) = R_W$$
with
\begin{align*}
a_0 (\nu) & = bx+cz- \nu (bz+cx) \\
d_0 (\nu) & = bz+cx \\
R_W & = \frac{(x^2 +z^2 +k^2 -xzk-4)(b^2 +c^2 +w^2 -bcw-4)}{4-w^2 }.
\end{align*}
Moreover, $R_W$ is $\tau_W$-invariant. For a fixed value of $-2 < w < 2 $, the left term of the equation is a quadratic function of $a$ and $d$ with positive coefficients. In the right term $R_W$ we recognize $(x,z,k)$ and $(b,c,w)$  as characters of representations in $\SU (2)$ of the free groups $\langle X , Z \rangle$ and $\langle C , A B^{-1} A^{-1} \rangle$ respectively, because $ XZ = ABCA = A^{-1} W A$ and $C A B^{-1} A^{-1} = W$. So according to Lemma \ref{frirel}, we have:
\begin{align*}
(x^2 +z^2 +k^2 -xzk-4) & \leq 0, \\
(b^2 +c^2 +w^2 -bcw-4) & \leq 0.
\end{align*}
Moreover $4 - w^2  > 0$, which implies $R_W \geq 0$. So the set of coordinates $d$ and $a$ satisfying the equation corresponds to an ellipse. This exhibits $\X_{\mathcal{C}} (M)$ as a family of ellipses $E_{\mathcal{C}} (M) (b,c)$ parametrized by $(b,c,k,x,z)$. Then we express the transformation $\tau_W$ on the coordinates $a$ and $d$ in terms of $a_0$ and $d_0$, which gives us :
\begin{equation} \left[ \begin{array}{c} a \\ d \end{array} \right] \mapsto \left[ \begin{array}{c} a_0(w) \\ d_0(u) \end{array} \right] + \left[ \begin{array}{cc} -1 & -w \\ w & w^2 -1 \end{array} \right] \cdot \left( \left[ \begin{array}{c} a \\ d \end{array} \right] - \left[ \begin{array}{c} a_0(w) \\ d_0(u) \end{array} \right] \right) . \end{equation}
 
This transformation is the rotation of angle $- \theta_W = -2 \cos ^{-1} (w/2)$ on the ellipse $E_{\mathcal{C}} (M) (x,z)$ defined earlier for fixed $(b,c,k,x,z)$. In particular, for fixed boundary traces $(b,c,k)$ and for almost all $(x,z)$, the angle $\theta_W$ is an irrational multiple of $\pi$. So for almost all $(x,z)$,  the action of $\tau_W$ is ergodic on the set $E_{\mathcal{C}} (M) (x,z)$.

\subsubsection{Conclusion}

Let $f : \X(M) \longrightarrow \mathbb{R}$ be a measurable $\Gamma_M$-invariant function. In particular $f$  is $\tau_T , \tau_U$ and $\tau_W$-invariant. We deduce from the Lemma of ergodic decomposition, that a measurable function $f$ that is $\tau_T$-invariant,  $\tau_U$-invariant or $\tau_W$-invariant, is almost everywhere equal to a function depending only on the coordinates $(b,c,k,a,x)$, the coordinates $(b,c,k,a,z)$ or the coordinates $(b,c,k,x,z)$ respectively. Therefore a measurable $\Gamma_M$-invariant function $f$ is almost everywhere equal to a function depending only on the traces of the boundary components $\mathcal{C} = (b,c,k)$. This proves the ergodicity of the $\Gamma_M$-action on $\X_\mathcal{C} (M)$, and hence Proposition \ref{thpp3}.

\subsection{The case of $N_{3,1}$}

In this section, we consider the case where $M$ is the non-orientable surface of genus 3 with one boundary component $N_{3,1}$. 
\begin{figure}[ht]\label{fig:N31}
\begin{center}
\scalebox{0.8} 
{
\begin{pspicture}(0,-2.936875)(14.424063,2.976875)
\psarc[linewidth=0.04](1.9040625,-0.576875){1.0}{90.0}{270.0}
\psline[linewidth=0.04cm](1.8840625,0.423125)(5.7040625,0.423125)
\psline[linewidth=0.04cm](1.9040625,-1.576875)(5.7440624,-1.556875)
\pscircle[linewidth=0.04,dimen=outer,fillstyle=crosshatch*,hatchwidth=0.04,hatchangle=0,hatchsep=0.14](2.1040626,-0.576875){0.45}
\pscircle[linewidth=0.04,dimen=outer,fillstyle=crosshatch*,hatchwidth=0.04,hatchangle=0,hatchsep=0.14](3.5640626,-0.596875){0.45}
\psellipse[linewidth=0.04,dimen=outer](7.6640625,-0.576875)(0.2,0.4)
\pscircle[linewidth=0.04,dimen=outer](11.984062,-0.536875){2.4}
\pscircle[linewidth=0.04,dimen=outer,fillstyle=crosshatch*,hatchwidth=0.04,hatchangle=0,hatchsep=0.14](10.844063,0.383125){0.45}
\pscircle[linewidth=0.04,dimen=outer,fillstyle=crosshatch*,hatchwidth=0.04,hatchangle=0,hatchsep=0.14](11.244062,-1.796875){0.45}
\pscircle[linewidth=0.04,dimen=outer,fillstyle=crosshatch*,hatchwidth=0.04,hatchangle=0,hatchsep=0.14](13.144062,0.503125){0.45}
\psbezier[linewidth=0.02](12.804063,-1.136875)(11.964063,-0.596875)(11.784062,1.103125)(11.144062,0.763125)
\psbezier[linewidth=0.02](10.424063,0.203125)(9.424063,-0.216875)(11.304063,-0.276875)(12.844063,-1.176875)
\psbezier[linewidth=0.02](12.844063,-1.216875)(11.484062,-0.976875)(9.884063,-0.676875)(10.904062,-1.496875)
\psbezier[linewidth=0.02](11.564062,-2.136875)(12.204062,-2.756875)(12.584063,-2.356875)(12.844063,-1.176875)
\psdots[dotsize=0.16](12.844063,-1.196875)
\psline[linewidth=0.04cm](12.544063,-1.676875)(12.584063,-1.996875)
\psline[linewidth=0.04cm](12.584063,-2.016875)(12.864062,-1.756875)
\psline[linewidth=0.04cm](10.784062,-0.536875)(10.6640625,-0.236875)
\psline[linewidth=0.04cm](10.684063,-0.256875)(11.004063,-0.236875)
\psline[linewidth=0.04cm](12.264063,0.003125)(12.364062,0.263125)
\psline[linewidth=0.04cm](12.384063,0.283125)(12.584063,0.043125)
\psline[linewidth=0.04cm](14.024062,-1.296875)(14.064062,-1.656875)
\psline[linewidth=0.04cm](14.064062,-1.656875)(14.404062,-1.476875)
\usefont{T1}{ptm}{m}{n}
\rput(13.010625,-2.046875){A}
\usefont{T1}{ptm}{m}{n}
\rput(10.41125,-0.546875){B}
\usefont{T1}{ptm}{m}{n}
\rput(12.21125,1.053125){C}
\usefont{T1}{ptm}{m}{n}
\rput(14.572657,-1.386875){K}
\pscircle[linewidth=0.04,dimen=outer,fillstyle=crosshatch*,hatchwidth=0.04,hatchangle=0,hatchsep=0.14](5.0440626,-0.596875){0.45}
\psbezier[linewidth=0.04](5.7040625,0.423125)(6.7040625,0.423125)(6.7040625,-0.216875)(7.6840625,-0.176875)
\psbezier[linewidth=0.04](7.6840625,-0.936875)(6.6840625,-0.916875)(6.7040625,-1.576875)(5.7240624,-1.556875)
\psbezier[linewidth=0.02](12.824062,-1.176875)(12.184063,0.563125)(12.244062,0.863125)(12.724063,0.663125)
\psbezier[linewidth=0.02](13.544063,0.343125)(14.064062,0.183125)(13.644062,-0.316875)(12.844063,-1.136875)
\end{pspicture} 
}
\end{center}
\caption{The surface $N_{3,1}$}
\end{figure}
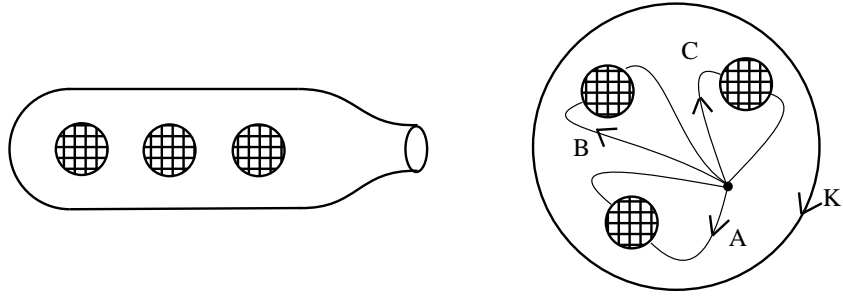

Its fundamental group admits the following presentation
$$\pi = \pi_1(M)=\langle A , B , C , K \mid  A^2 B^2 C^2 K^{-1} \rangle ,$$
where $A,B,C,K$ are the curves drawn in Figure \ref{fig:N31}. We see that $\pi$ is a free group on three generators $A,B,C$, so the coordinates on the space $\X (M)$ are given by the seven trace functions $a,b,c,d,x,y,z$ satisfying the Fricke relation (\ref{fricke}). The character of the boundary of $M$ is $\X (\partial M) = \{ k \in [-2 , 2 ] \} \, ,$ where $k=\tr(\rho(K)) = \tr (\rho (A^2B^2C^2)) = abcd-bcy-acz-bax+a^2+b^2+c^2-2$.

\begin{proposition}\label{thn31} For all boundary component $\mathcal{C} = k \in ] -2 , 2 [$, the action of $\Gamma_M$ on $\X_{\mathcal{C}} (M)$ is ergodic.
\end{proposition}

\begin{proof}
The curve $U$ shown in Figure 9 is represented by $W=AACC \in \pi_1(M)$. It is a simple two-sided curve, so the Dehn twist can be defined about it.
\begin{figure}[ht]
\begin{center}
\scalebox{1} 
{
\begin{pspicture}(0,-2.45)(5.323125,2.41)
\pscircle[linewidth=0.04,dimen=outer](2.56,0.01){2.4}
\pscircle[linewidth=0.04,dimen=outer,fillstyle=crosshatch,hatchwidth=0.04,hatchangle=0,hatchsep=0.14](1.42,0.93){0.45}
\pscircle[linewidth=0.04,dimen=outer,fillstyle=crosshatch,hatchwidth=0.04,hatchangle=0,hatchsep=0.14](1.82,-1.25){0.45}
\pscircle[linewidth=0.04,dimen=outer,fillstyle=crosshatch,hatchwidth=0.04,hatchangle=0,hatchsep=0.14](3.72,1.05){0.45}
\psbezier[linewidth=0.02](3.38,-0.59)(2.54,-0.05)(2.36,1.65)(1.72,1.31)
\psbezier[linewidth=0.02](1.0,0.75)(0.0,0.33)(1.88,0.27)(3.42,-0.63)
\psbezier[linewidth=0.02](3.42,-0.67)(2.06,-0.43)(0.46,-0.13)(1.48,-0.95)
\psbezier[linewidth=0.02](2.14,-1.59)(2.78,-2.21)(3.16,-1.81)(3.42,-0.63)
\psdots[dotsize=0.16](3.42,-0.65)
\psline[linewidth=0.04cm](4.6,-0.75)(4.64,-1.11)
\psline[linewidth=0.04cm](4.64,-1.11)(4.98,-0.93)
\usefont{T1}{ptm}{m}{n}
\rput(5.148594,-0.84){K}
\psbezier[linewidth=0.02](3.4,-0.63)(2.76,1.11)(2.82,1.41)(3.3,1.21)
\psbezier[linewidth=0.02](4.12,0.89)(4.64,0.73)(4.22,0.23)(3.42,-0.59)
\psbezier[linewidth=0.08,linestyle=dashed,dash=0.16cm 0.16cm](1.58,0.15)(2.92,0.29)(2.44,2.27)(3.6,1.83)(4.76,1.39)(4.739589,0.19755611)(4.26,-0.77)(3.7804105,-1.7375561)(2.52,-2.41)(1.68,-1.99)(0.84,-1.57)(0.24,0.01)(1.58,0.15)
\usefont{T1}{ptm}{m}{n}
\rput(2.5295312,1.95){\Large U}
\psline[linewidth=0.04cm](2.46,1.55)(2.8,1.61)
\psline[linewidth=0.04cm](2.84,1.61)(2.88,1.23)
\end{pspicture} 
}
\end{center}
\caption{The curve $U$}
\end{figure}
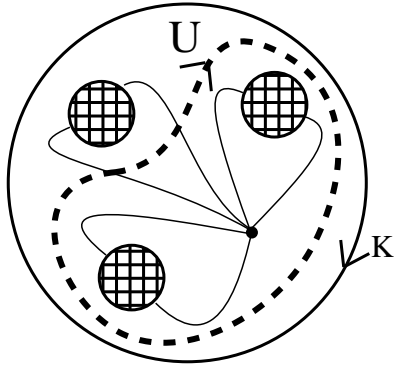

The Dehn twist $\tau_U$ about the curve $U$ is given by the following automorphism of $\pi_1(M)$:
\begin{align*}
A \mapsto & \,\, A \\
B \mapsto & \,\, A^{-2}C^{-2}BC^2A^2 \\
C \mapsto & \,\, C
\end{align*}
The elements corresponding to $X,Y,Z,D$ and $K$ are transformed by $\tau_U$ as follows:
\begin{align*}
X = AB \mapsto & \,\, A^{-1}C^{-2}BC^2AA \\
Y = BC \mapsto & \,\, A^{-2}C^{-2}BC^2A^2C\\
Z = CA \mapsto & \,\, CA \\
D = ABC \mapsto & \,\, A^{-1}C^{-2}BC^2A^2C \\
K = AABBC \mapsto & \,\, AABBC
\end{align*}

This transformation leaves invariant the boundary component $K$. So the Dehn twist $\tau_U$ induces an action on $\X_{\mathcal{C}} (M)$ which can be seen on the coordinates $(a,b,c,x,y,z,d) \in ]-2 , 2 [^7$ as:
\begin{align*}
a \mapsto & a \\
b \mapsto & b \\
c \mapsto & c \\
z \mapsto & z \\
x \mapsto & c^3d-c^2yz-c^2x-2cd+cbz+cay+x \\
y \mapsto & \tau_U (y)\\
d \mapsto & \tau_U (d) 
\end{align*}
where \\
$\tau_U (y) = -a^4bc^3+a^4bc+a^4c^2y+a^3bc^4z-a^3bc^2z+a^3c^3x-a^3c^3yz-a^3c^2d-2a^3cx+a^3d-a^2bc^5-a^2bc^3z^2+4a^2bc^3+a^2bcz^2-2a^2bc-a^2c^4xz+a^2c^4y+a^2c^3dz+3a^2c^2xz+a^2c^2yz^2-3a^2c^2y-2a^2cdz-a^2xz-a^2y+abc^4z-3abc^2z+abz+ac^5x-ac^4d-5ac^3x-ac^3yz+4ac^2d+5acx+2acyz-2ad+y$\\
and\\
$\tau_U (d) = -a^3bc^3+a^3bc+a^3c^2y+a^2bc^4z-a^2bc^2z+a^2c^3x-a^2c^3yz-a^2c^2d-2a^2cx+a^2d-abc^5-abc^3z^2+3abc^3+abcz^2-abc-ac^4xz+ac^4y+ac^3dz+3ac^2xz+ac^2yz^2-2ac^2y-2acdz-axz-ay+bc^4z-2bc^2z+bz+c^5x-c^4d-4c^3x-c^3yz+3c^2d+3cx+cyz-d$.\\

We define new coordinates on $\X (M)$ by replacing $d$ with its expression in function of $a,b,c,x,y,z$ and $k$.
\begin{equation}
d = \frac{bcy+acz+abx-a^2-b^2-c^2 +2+k}{abc} .
\end{equation}
 The equation (\ref{fricke}) becomes
 \begin{equation}\label{fri3}
\left( \frac{x}{c} \right)^2 + (acz-a^2-c^2+2) \left( \frac{x}{c}\frac{y}{a} \right) + \left( \frac{y}{a} \right)^2 + 2 D x +2Ey +F = 0 , 
\end{equation}
with constants $D,E,F$ depending on the $\tau_U$-invariant coordinates $a,b,c,z,k$. We make a change of variable $y' = \frac{y}{a}$ and $x' = \frac{x}{c} $and the equation $(\ref{fri3})$ becomes
\begin{equation}
x'^2 + (acz-a^2-c^2+2)x' y' + y'^2 + 2 D' x +2E'y +F = 0 , 
\end{equation}
with\\

$D' = \dfrac{1}{abc}(a^2c^2-2a^2-ab^2cz-ac^3z+2acz+b^2c^2-2b^2+c^4-c^2k-4c^2+2k+4)$\\

$E' = \dfrac{1}{abc}(a^4-a^3cz+a^2b^2+a^2c^2-a^2k-4a^2-ab^2cz+2acz-2b^2-2c^2+2k+4)$\\

$F = \dfrac{1}{c^2b^2a^2}(a^4+a^3b^2cz-2a^3cz+a^2b^2c^2k-2a^2b^2c^2+2a^2b^2+a^2c^2z^2+2a^2c^2-2a^2k-4a^2+ab^4cz+ab^2c^3z-ab^2ckz-4ab^2cz-2ac^3z+2ackz+4acz+b^4+2b^2c^2-2b^2k-4b^2+c^4-2c^2k-4c^2+k^2+4k+4)$.\\

We denote $u = \tr (\rho (U)) = acz - a^2 - c^2 +2$ and when $u \neq \pm 2$ we rewrite \ref{fri3} as
$$Q_u ( x-x_0 (u) , y - y_0(u) ) = R ,$$
where\\
$x'_0 = \dfrac{1 }{abc(u^2-4)}(-a^4u+a^3cuz-a^2b^2u-a^2c^2u+2a^2c^2+a^2ku+4a^2u-4a^2+ab^2cuz-2ab^2cz-2ac^3z-2acuz+4acz+2b^2c^2+2b^2u-4b^2+2c^4-2c^2k+2c^2u-8c^2-2ku+4k-4u+8)$\\

$y'_0 = \dfrac{1 }{abc(u^2-4)}(2a^4-2a^3cz+2a^2b^2-a^2c^2u+2a^2c^2-2a^2k+2a^2u-8a^2+ab^2cuz-2ab^2cz+ac^3uz-2acuz+4acz-b^2c^2u+2b^2u-4b^2-c^4u+c^2ku+4c^2u-4c^2-2ku+4k-4u+8)$\\

and 
$$ R = \dfrac{(a^2 +c^2 +z^2 -acz-4)((b^2 -2)^2 +u^2 +k^2 -(b^2 -2)uk -4)}{b^2 (4-u^2 )} .$$

The term $R$ is also $\tau_U$-invariant. For a particular value of $-2 < u < 2 $, the left term of the equation is a quadratic function on $(x',z')$ with positive coefficients. In the right term $R$ we recognize that $(a,c,z)$ and $(b^2 -2,u,k)$  are the characters of representations in $\SU (2)$ of the free groups $\langle A, C \rangle$ and $\langle BB , CCAA \rangle$ respectively, as we have $ AC = Z$ and $CCAABB = K $. So according to Fricke's relation, we have:
\begin{align}
&(a^2 +c^2 +z^2 -acz-4) < 0, \\
&((b^2 -2)^2 +u^2 +k^2 -(b^2 -2)uk-4) < 0.
\end{align}
Moreover $4 - u^2 > 0$, so we deduce that $R>0$. So the set of coordinates $(x',y')$ satisfying the equation corresponds to an ellipse. This exhibits $\X_{\mathcal{C}} (M)$ as a family of ellipses $E_{\mathcal{C}} (M) (a,b,c,z)$ that are parametrized by $(k,a,b,c,z)$. Now we express the action of $\tau_U$ on $x',y'$ using $x'_0 , y'_0$, which gives us :
$$ \left[ \begin{array}{c} x' \\ y' \end{array} \right] \mapsto
 \left[ \begin{array}{c} x'_0(u) \\ y'_0(u) \end{array} \right] +
\left[ \begin{array}{cc} -1 & -u \\ (u^2 - 1) & u \end{array} \right] \cdot
\left( \left[ \begin{array}{c} x' \\ y' \end{array} \right] -
 \left[ \begin{array}{c} x'_0(u) \\ y'_0(u) \end{array} \right] \right) .$$
 
 This transformation is a rotation of angle $2 \theta_U = -2 \cos ^{-1} (U/2)$ on the ellipse $E_{\mathcal{C}} (M) (a,b,c,z)$ defined for fixed $(k,a,b,c,z)$. In fact, for any boundary trace $k$ and for almost all $(a,b,c,z)$, $\theta_U$ is an irrational multiple of $\pi$. So for almost all $(a,b,c,z)$, the action of $\tau_U$ is ergodic on the set $E_{\mathcal{C}} (M) (a,b,c,z)$. Finally, let $f : \X(M) \rightarrow \R $ be a $\tau_U$-invariant measurable function. We deduce from the Lemma of ergodic decomposition, that there exists a function $H \, : \, [-2,2]^5 \longrightarrow \mathbb{R}$ such that $f([\rho]) = H(k,a,b,c,z)$ almost everywhere.
 
Finally, we split the surface along the curve $X=AB$ to obtain a surface $A$ which is a three holed projective plane whose boundary components are $K$ and $X_-, X_+$ which correspond to the two sides of $X$. According to the proposition $\ref{thpp3}$, a $\Gamma_A$-invariant measurable function $f : \X (A) \longrightarrow \mathbb{R}$ which is , is almost everywhere equal to a function depending only on the boundary character $(x,x,k)$. So we have $f([\rho]) = G(k,x)$ almost everywhere. Combining the two arguments proves that a $\Gamma_M$-invariant function $f$ depends only on the traces of the boundary component $k$. This ends the proof of Proposition \ref{thn31}.
\end{proof}

\section{Surface of odd genus}\label{section:conclusion}

In this section, we prove Theorem \ref{thm:open} in the case of a non-orientable surface of odd genus $k$. We split the surface $M$ along a separating curve, such that one of the subsurface is orientable, and the other is a non-orientable surface of Euler characteristic $-2$. Two cases occur depending on the genus.

\subsection{When the genus is greater than 3}
Let $M$ be the non-orientable surface $N_{2g+1,m}$ with $g \geq 1$ and $\chi (M) < -2$. Let $C$ be a separating circle such that one of the two subsurfaces is the surface $A = N_{3,1}$ and the other is the orientable surface $B = \Sigma_{g,m+1}$ of genus $g$ with $m+1$ boundary components.

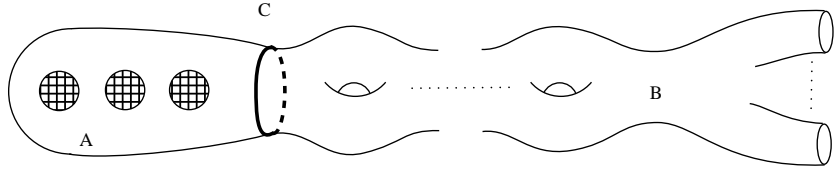
\begin{figure}[ht]
\begin{center}
\scalebox{0.6} 
{
\begin{pspicture}(0,-2.221875)(19.217188,2.246875)
\psarc[linewidth=0.03](2.3271875,-0.496875){1.39}{90.0}{270.0}
\pscircle[linewidth=0.03,dimen=outer,fillstyle=crosshatch*,hatchwidth=0.04,hatchangle=0,hatchsep=0.14](3.5171876,-0.466875){0.45}
\pscircle[linewidth=0.03,dimen=outer,fillstyle=crosshatch*,hatchwidth=0.04,hatchangle=0,hatchsep=0.14](4.9371877,-0.466875){0.45}
\psarc[linewidth=0.03](8.587188,-0.696875){0.39}{25.0}{155.0}
\psarc[linewidth=0.03](8.6171875,0.273125){0.88}{-140.0}{-40.0}
\psarc[linewidth=0.03](13.147187,-0.696875){0.39}{25.0}{155.0}
\psarc[linewidth=0.03](13.177188,0.273125){0.88}{-140.0}{-40.0}
\psbezier[linewidth=0.03](2.3171875,0.893125)(3.271238,0.9542026)(5.6276937,0.80022335)(6.6014915,0.5066742)(7.575289,0.213125)(7.911793,1.1232382)(8.807441,0.917643)(9.703088,0.7120478)(9.622504,0.4292545)(10.457188,0.413125)
\psbezier[linewidth=0.03](11.417188,0.433125)(12.1171875,0.413125)(12.417529,1.079247)(13.397187,0.913125)(14.376846,0.747003)(14.500669,0.3364951)(15.397187,0.393125)(16.293707,0.4497549)(16.997187,1.373125)(19.017187,1.293125)
\psbezier[linewidth=0.03](2.2771876,-1.886875)(3.2171874,-2.046875)(5.6771874,-1.726875)(6.6171875,-1.466875)(7.5571876,-1.206875)(7.8971877,-2.086875)(8.807441,-1.862357)(9.717694,-1.637839)(9.662504,-1.3507454)(10.477187,-1.366875)
\psbezier[linewidth=0.03](11.437187,-1.386875)(12.1171875,-1.306875)(12.437529,-1.9607531)(13.397187,-1.806875)(14.356846,-1.652997)(14.460669,-1.1835049)(15.317187,-1.186875)(16.173706,-1.1902452)(16.757187,-2.206875)(19.017187,-2.126875)
\psline[linewidth=0.04cm,linestyle=dotted,dotsep=0.16cm](9.857187,-0.426875)(12.037188,-0.406875)
\psellipse[linewidth=0.03,dimen=outer](19.047188,0.823125)(0.17,0.47)
\psellipse[linewidth=0.03,dimen=outer](19.007187,-1.666875)(0.17,0.46)
\psbezier[linewidth=0.03](19.037188,0.353125)(18.557188,0.413125)(18.017187,0.173125)(17.477188,0.053125)
\psbezier[linewidth=0.03](18.977188,-1.206875)(18.117188,-1.166875)(17.737188,-0.826875)(17.357187,-0.766875)
\psline[linewidth=0.04cm,linestyle=dotted,dotsep=0.16cm](18.737188,0.173125)(18.717188,-0.886875)
\pscircle[linewidth=0.03,dimen=outer,fillstyle=crosshatch*,hatchwidth=0.04,hatchangle=0,hatchsep=0.14](2.0571876,-0.486875){0.45}
\psbezier[linewidth=0.08](6.772197,0.473125)(6.5472426,0.513125)(6.41736,0.09292267)(6.4171877,-0.426875)(6.417015,-0.9466727)(6.4272676,-1.666875)(6.8171873,-1.406875)
\psbezier[linewidth=0.08,linestyle=dashed,dash=0.16cm 0.16cm](6.7371874,0.493125)(6.9771876,0.453125)(7.0571876,-0.046875)(7.0571876,-0.486875)(7.0571876,-0.926875)(6.9571877,-1.446875)(6.7971873,-1.426875)
\usefont{T1}{ptm}{m}{n}
\rput(2.66375,-1.556875){A}
\usefont{T1}{ptm}{m}{n}
\rput(15.284375,-0.536875){B}
\usefont{T1}{ptm}{m}{n}
\rput(6.604375,1.323125){C}
\end{pspicture} 
}
\end{center}
\caption{Decomposition of $N_{2g+1,m}$}
\end{figure}

 Let $f : \X(M) \rightarrow \R$ be a measurable $\Gamma_M$-invariant function. The Dehn twist $\tau_C$ about the curve $C$ acts on the generic fiber of the application $j : \X (M) \rightarrow \X(M;A,B,C)$ as the rotation of angle $2 \theta_C = 2 \cos^{-1} (\tr (\rho (C)))$ for any representation $\rho$ in the fiber $j^{-1}( [\rho_A] , [\rho_B] )$.  For almost all $( [\rho_A] , [\rho_B] ) \in \X(M;A,B,C)$, the angle is an irrational multiple of $\pi$ and the $\tau_C$-action is ergodic. By the lemma of ergodic decomposition, there exists a measurable function $h : \X(M;A,B,C) \rightarrow \R$ such that $f = h \circ j$ almost everywhere. There are natural injective maps $\Gamma_A \hookrightarrow \Gamma_M$ and $\Gamma_B \hookrightarrow \Gamma_M$. Hence, the function $h$ is $\Gamma_A$-invariant and $\Gamma_B$-invariant.

Next, consider the projection 
\begin{align*}
\phi : \X (M ; A,B,C) & \longrightarrow \X (B) \\
([\rho_A ] , [\rho_B ] ) \longmapsto [\rho_B ].
\end{align*}
The fiber $\phi^{-1} ([\rho_B])$ can be identified with $\X_c (A)$ where $c = \tr (\rho_B (C))$. According to Proposition \ref{thn31}, the mapping class group $\Gamma_A$ acts ergodically on the fibers of $\phi$. Thus, by the lemma of ergodic decomposition, there exists a measurable function $H : \X (B) \rightarrow \R$ such that $h = H \circ \phi$ almost everywhere. Moreover, the function $H$ is $\Gamma_B$-invariant.

We infer from the ergodicity result in the orientable case, that a $\Gamma_B$-invariant function is almost everywhere constant on almost every level set of the application 
\begin{align*}
\partial^{\sharp} : \X (B) & \longrightarrow [-2 , 2]^{m+1} \\
[\rho] & \longmapsto (\tr (\rho (k_1)) , ... , \tr (\rho (k_m)), \tr (\rho (C) ) ).
\end{align*}

So, there exists a measurable function $F : [-2 , 2]^{m+1} \rightarrow \R$ such that  $$f = F \circ \delta^{\sharp} \circ \phi \circ j = F (\tr (\rho (k_1)) , ... , \tr (\rho (k_m)), \tr (\rho (C) ) ) \mbox{ almost everywhere.}$$ 

The orientable surface $B$ has negative Euler characteristic $\chi (B) \leq -1$ and hence $B$ can be decomposed into pants. The surface $A$ can be decomposed into a pair of pants and a 2-holed projective plane such that the curve $C$ is a boundary component of the pair of pants. Gluing the two pairs of pants containing $C$ as a boundary component gives a 4-holed sphere $S$, which is embedded in $M$. The circle $C$ is essential in the surface $S$. The ergocity result in the case of a 4-holed sphere shows that a $\Gamma_S$-invariant function is almost everywhere equal to a function that does not depend on the trace $\tr (\rho (C))$. The function $f$ is $\Gamma_S$-invariant, and hence is almost everywhere equal to a function depending only on the traces of the boundaries $\mathcal{C} = (k_1 , \dots , k_m)$, which proves ergodicity in this case.

\subsection{In genus 1}.
Let $M$ be the non-orientable surface $N_{1,m}$ with $m > 3$, and let $C$ be a separating circle such that one of the two subsurfaces is the surface $A = \N_{1,3}$ and the other is a $(m+1)$-holed sphere $B = \Sigma_{0,m+1}$. 

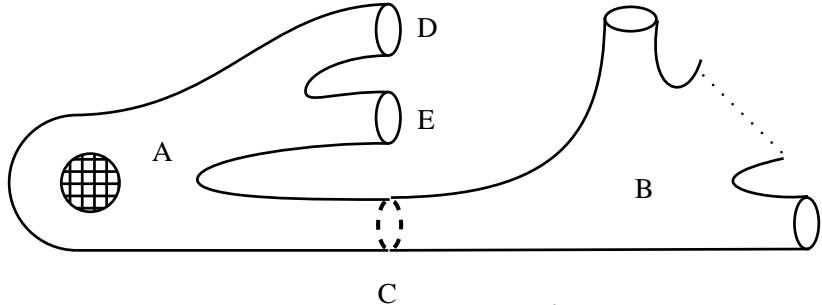
\begin{figure}[ht]
\begin{center}
\scalebox{0.9} 
{
\begin{pspicture}(0,-2.266875)(12.88375,2.306875)
\psarc[linewidth=0.04](1.90375,-1.246875){1.0}{90.0}{270.0}
\pscircle[linewidth=0.04,dimen=outer,fillstyle=crosshatch*,hatchwidth=0.04,hatchangle=0,hatchsep=0.14](2.10375,-1.246875){0.45}
\psellipse[linewidth=0.04,dimen=outer](6.50375,1.013125)(0.2,0.4)
\psellipse[linewidth=0.04,dimen=outer](6.50375,-0.286875)(0.2,0.4)
\psbezier[linewidth=0.04](1.90375,-0.246875)(4.0306315,-0.206875)(4.6736426,1.353125)(6.50375,1.393125)
\psbezier[linewidth=0.04](6.4840484,0.633125)(5.90375,0.633125)(5.32375,0.393125)(5.28375,0.133125)(5.24375,-0.126875)(5.72375,0.113125)(6.50375,0.093125)
\psbezier[linewidth=0.04](6.50375,-0.666875)(5.46375,-0.646875)(3.7248688,-0.87354165)(3.6843095,-1.1902083)(3.64375,-1.506875)(5.4892044,-1.4957639)(6.5234704,-1.4957639)
\psline[linewidth=0.04cm](1.90375,-2.246875)(6.50375,-2.246875)
\psellipse[linewidth=0.06,linestyle=dashed,dash=0.16cm 0.16cm,dimen=outer](6.52375,-1.866875)(0.2,0.4)
\rput{-90.0}(8.910625,11.256875){\psellipse[linewidth=0.04,dimen=outer](10.08375,1.173125)(0.2,0.4)}
\psbezier[linewidth=0.04](6.54375,-1.466875)(8.96375,-1.446875)(9.68375,-0.586875)(9.70375,1.193125)
\psbezier[linewidth=0.04](10.48375,1.153125)(10.30375,0.073125)(10.92375,-0.126875)(11.12375,0.573125)
\psline[linewidth=0.04cm,linestyle=dotted,dotsep=0.16cm](11.14375,0.353125)(12.32375,-0.806875)
\psellipse[linewidth=0.04,dimen=outer](12.68375,-1.846875)(0.2,0.4)
\psline[linewidth=0.04cm](6.52375,-2.246875)(12.72375,-2.226875)
\psbezier[linewidth=0.04](12.70375,-1.446875)(12.00375,-1.526875)(10.80375,-1.226875)(12.34375,-0.886875)
\usefont{T1}{ptm}{m}{n}
\rput(3.1703124,-0.776875){A}
\usefont{T1}{ptm}{m}{n}
\rput(10.290937,-1.336875){B}
\usefont{T1}{ptm}{m}{n}
\rput(7.08375,1.043125){D}
\usefont{T1}{ptm}{m}{n}
\rput(7.065781,-0.316875){E}
\usefont{T1}{ptm}{m}{n}
\rput(6.4909377,-2.876875){C}
\end{pspicture} 
}
\end{center}
\caption{Decomposition of $N_{1,m}$}
\end{figure}

The proof uses the same arguments as in the previous case, the only difference is that we have to keep track of the two other boundary components of the subsurface $A$. Let $f : \X(M) \rightarrow \R$ be a measurable $\Gamma_M$-invariant function. The action of the Dehn twist about the curve $C$ is ergodic on almost all fibers of $j : \X (M) \rightarrow \X(M;A,B,C)$. The mapping class group  $\Gamma_A$ acts ergodically on almost every fibers of the map
\begin{align*}
\phi ' : \X (M ; A,B,C) & \longrightarrow \X (B) \times [-2 , 2]^2 \\
([\rho_A ] , [\rho_B ] ) & \longmapsto ([\rho_B ] , \tr (\rho_A ( D)) , \tr (\rho_A (E )) ).
\end{align*}
The mapping class group $\Gamma_B$ acts ergodically on almost every fiber of the map
\begin{align*}
\widetilde{\partial^{\sharp} }: \X (B) \times [-2 , 2]^2 & \longrightarrow [-2 , 2]^{m+3} \\
([\rho] , d , e) & \longmapsto (\tr (\rho (k_1)) , ... , \tr (\rho (k_m)), \tr (\rho (C) ),d,e ).
\end{align*}
Hence there exists a measurable function $F : [-2 , 2]^{m+3} \rightarrow \R$ such that  $f = F \circ \widetilde{\delta^{\sharp}} \circ \phi ' \circ j = F (\tr (\rho (k_1)) , ... , \tr (\rho (k_m)), \tr (\rho (C) ) ,\tr (\rho (D)) , \tr (\rho (E)))$ almost everywhere. As in the previous case we can find a 4-holed sphere $S$ embedded in $M$ such that $C$ is an essential circle in $S$. Hence the function $F$ does not depend on $\tr (\rho (C))$ and the proof of Theorem \ref{thm:open} is complete.
$\qed$



\begin{thebibliography}{}

\bibitem{freed} \textsc{Freed D.},
\emph{Reidemeister torsion, spectral sequences, and Brieskorn spheres},
J. Reine Angew. Math. , {\bf 429} (1992), 75--89.

\bibitem{frobenius} \textsc{Frobenius G., Schur I.}, \textit{\"Uber die reellen Darstellung der endlichen Gruppen}, Sitzungsberichte der k\"oniglich preussischen Akademie der Wissenschaften (1906), 186--208.

\bibitem{furst} \textsc{Furstenberg H.},
\emph{Recurence in Ergodic Theory and Combinatorial Number Theory},
Princeton University Press (1981)

\bibitem{gendulphe} \textsc{Gendulphe M.}, \textit{Paysage Systolique Des Surfaces Hyperboliques Compactes De Caracteristique -1},  arXiv:math/0508036.

\bibitem{goldman84} \textsc{Goldman W.},
\emph{The symplectic nature of fundamental groups of surfaces},
Adv. Math. , {\bf 85} (1984),200--225.

\bibitem{Goldman86} \textsc{Goldman W.},
\emph{Invariant functions on Lie groups and Hamiltonian flows of
surface group representations},
Invent.\ Math.\ {\bf 85} (1986), no.\ 2, 263--302.

\bibitem{goldman88} \textsc{Goldman W.},
\emph{Topological components of spaces of representations},
Invent.\ Math.\ {\bf 93} (1988), 557--607.

\bibitem{go1} \textsc{Goldman W.},
\emph{Ergodic theory on moduli spaces},
Ann.\ of Math.\ (2) {\bf 146} (1997), no.\ 3, 475--507.

\bibitem{goldman04} \textsc{Goldman W.},
\emph{The complex-symplectic geometry of
$SL(2,\mathbb{C})$-characters over surfaces},
Algebraic groups and arithmetic, 375--407,
Tata Inst.\ Fund.\ Res., Mumbai, 2004.

\bibitem{goldman06} \textsc{Goldman W.},
\emph{Mapping class group dynamics on surface group representations} In: Problems in Mapping Class Groups and Related Topics. Proceedings of Symposia in Pure Math. Amer. Math. Soc. (to appear).

\bibitem{jefwei} \textsc{Jeffrey L., Weitsman J.}, \textit{Bohr-Sommerfeld orbits in the moduli space of flat connections and the Verlinde dimension formula}, Commun. Math. Phys. \textbf{150} (1992), 595--620.

\bibitem{ho} \textsc{Ho N., Jeffrey L.}, \textit{The volume of the moduli space of flat connections on a nonorientable 2-manifold}, Commun. Math. Phys. \textbf{256} (2005), 539-564.


\bibitem{klein} \textsc{Klein D.}, \textit{Goldman flows on non-orientable surfaces}, arXiv:0710.5265.

\bibitem{korkmaz} \textsc{Korkmaz M.}, \textit{Mapping class group of nonorientable surfaces}, Geom. Dedicata \textbf{89} (2002) 109--133.

\bibitem{lickorish} \textsc{Lickorish W.B.R.}, \textit{Homeomorphisms on non-orientable two-manifolds}, Math. Proc. Cambridge Philos. Soc. \textbf{59} (1963), 307--317.

\bibitem{lickorish2} \textsc{Lickorish W.B.R.}, \textit{A finite set of generators for the homeotopy group of a 2-manifold}, Math. Proc. Cambridge Philos. Soc. \textbf{60} (1964), 769--778.

\bibitem{magnus} \textsc{Magnus W.}, \textit{Rings of Fricke characters and automorphism group of free groups}, Math. Zeitschrift \textbf{170} (1980), 91--103.

\bibitem{mangler} \textsc{Mangler W.}, \textit{Die Klassen topologisher Abbildungen einer geschlossenen Flache auf sich}, Math. Zeitschrift \textbf{44} (1939), 541--554.

\bibitem{mulase} \textsc{Mulase M., Penkava M.}, \textit{Volume of representation varieties}, arXiv:math/0212012 


\bibitem{nielsen} \textsc{Nielsen J.}, \textit{Die isomorphismen der allgemeinen unendlichen Gruppen mit zwei Erzeugenden}, Math. Ann. \textbf{78} (1918) , 385--397.

\bibitem{picxia1} \textsc{Pickrell D., Xia E.}, \textit{Ergodicity of mapping class group actions on representation variety, I. Closed surfaces}, Comment. Math. Helv. \textbf{77} (2001), 339--362.

\bibitem{picxia2} \textsc{Pickrell D., Xia E.}, \textit{Ergodicity of mapping class group actions on representation variety, II. Surfaces with boundaries}, Transformation Groups \textbf{8} (2003), 397--402.

\bibitem{prexia1} \textsc{Previte J., Xia E.}, \textit{Topological dynamics on Moduli Spaces I}, Pac. J. Math. \textbf{193} (2000), 397--417.

\bibitem{prexia2} \textsc{Previte J., Xia E.}, \textit{Topological dynamics on Moduli Spaces II}, Trans. Amer. Math. Soc. \textbf{354} (2002), 2475--2494.

\bibitem{prexia3} \textsc{Previte J., Xia E.}, \textit{Exceptional discrete mapping class group orbits in moduli spaces}, Forum Math. \textbf{15} (2003), 949-954.

\bibitem{stukow} \textsc{Stukow M.}, \textit{The twist subgroup of the mapping class group of a nonorientable surface}, arXiv:0709.2798

\bibitem{witten} \textsc{Witten E.}, \textit{On quantum gauge theory in two dimensions}, Commun. Math. Phys. \textbf{141} (1991), 153--209.

\end{thebibliography}
\end{document}